\documentclass[a4paper]{amsart}

\usepackage{amsmath,amssymb,amsthm}
\usepackage[mathcal]{euscript}

\usepackage[latin1]{inputenc}

\input cyracc.def \font\tencyr=wncyr10 \def\russe{\tencyr\cyracc} 
\def\Sha{\text{\russe{Sh}}}

\newtheorem{theorem}{Theorem}
\newtheorem{definition}[theorem]{Definition}
\newtheorem{proposition}[theorem]{Proposition}
\newtheorem{lemma}[theorem]{Lemma}
\newtheorem{corollary}[theorem]{Corollary}
\newtheorem{conjecture}{Conjecture}

\renewcommand{\geq}{\geqslant}
\renewcommand{\leq}{\leqslant}

\newcommand{\Q}{\mathbb Q}
\newcommand{\N}{\mathbb N}
\newcommand{\F}{\mathbb F}

\newcommand{\Z}{\mathbb Z}

\newcommand{\ssss}{\scriptscriptstyle}
\newcommand{\sss}{\scriptstyle}

\newcommand\qbi[3]{{{#1}\atopwithdelims[]{#2}}_{#3}}
\newcommand\bi[2]{{{#1}\atopwithdelims(){#2}}}
\newcommand{\la}{\lambda}

\DeclareMathOperator{\auts}{Aut^{\ssss {\it S}}}
\DeclareMathOperator{\aut}{Aut}
\DeclareMathOperator{\uprob}{{\it u}-prob}
\DeclareMathOperator{\m}{M_{\sss {\it u}}}
\DeclareMathOperator{\mzero}{M_{\sss {0}}}
\DeclareMathOperator{\mun}{M_{\sss { 1}}}
\DeclareMathOperator{\ms}{M_{\sss {\it u}}^{\ssss {\it S}}}
\DeclareMathOperator{\mszero}{M_{\sss {0}}^{\ssss {\it S}}}
\DeclareMathOperator{\msun}{M_{\sss {1}}^{\ssss {\it S}}}

\DeclareMathOperator{\cl}{C\ell}
\DeclareMathOperator{\rk}{rk}

\DeclareMathOperator{\hominj}{Hom_{inj}} 
 
\DeclareMathOperator{\sel}{Sel} 

\begin{document}  
    
\title{$p^\ell$-torsion points in finite abelian groups and combinatorial identities}

\author{Christophe Delaunay}
\address{Universit\'e de Franche-Comt\'e \\
Laboratoire de Math\'ematiques de Besan\c con, CNRS UMR UMR 6623 \\
Facult\'es des Sciences et Techniques \\ 
16 route de Gray \\
25030 Besan\c con  \\
France.}
\email{christophe.delaunay@univ-fcomte.fr}
\author{Fr\'ed\'eric Jouhet}
\address{Universit\'e de Lyon \\
CNRS  \\
Universit\'e Lyon 1 \\
Institut Camille Jordan \\
43, boulevard du 11 novembre 1918 \\
F-69622 Villeurbanne Cedex \\
France.}
\email{jouhet@math.univ-lyon1.fr}
\date{\today}

\begin{abstract}
The main aim of this article is to compute all the moments of the number of $p^\ell$-torsion elements in some type of finite abelian groups. The averages involved in these moments are those defined for the Cohen-Lenstra heuristics for class groups and their adaptation for Tate-Shafarevich groups. In particular, we prove that the heuristic model for Tate-Shafarevich groups is compatible with the recent conjecture of Poonen and Rains about the moments of the orders of $p$-Selmer groups of elliptic curves. For our purpose, we are led to define  certain polynomials indexed by integer partitions and to study them in a combinatorial way. Moreover, from our probabilistic model, we derive combinatorial identities, some of which appearing to be new, the others being related to the theory of symmetric functions. In some sense, our method therefore gives for these identities a somehow natural algebraic context.
\end{abstract}

\maketitle

%
%

\section{Introduction}

In this work, we compute  the averages, in some sense, of a large class of functions defined over some families of finite abelian 
groups. Theses averages come from the well-known Cohen-Lenstra heuristics model for understanding the behavior  of class 
groups of number fields and their adaptation for Tate-Shafarevich\footnote{Note that the original heuristic assumption for Tate-Shafarevich groups 
in \cite{delaunay1} 
must be modified a little bit; this correction is explained in section~\ref{heur_sha}.} groups of elliptic curves 
(see \cite{cohen-lenstra, cohen-martinet, delaunay1,delaunay2}). In particular, our computations lead to 
a conjectural formula for all the moments of the number of $n$-torsion points in Tate-Shafarevich groups. We prove that our
prediction is consistent  with the recent model of Poonen and Rains that led them to conjecture a formula for all the moments of the 
number of $n$-torsion points in the Selmer groups of elliptic curves if $n$ is squarefree (see \cite{poonen-rains}). 
Furthermore, our approach allows us to conjecture a formula for the moments of the number of $n$-torsion points in the Selmer groups of elliptic curves for 
all positive integer $n$ (not only squarefree).   

The computation of our averages involves several combinatorial tools. We introduce  multivariate polynomials
$R_\lambda$, indexed by integer partitions $\lambda$, depending on a fixed complex parameter $t$ and we investigate
combinatorial properties satisfied by them. As $\lambda$ runs through all the set of integer partitions with parts bounded by a fixed integer $\ell$, our polynomials form a
basis of the polynomial ring $\Z(t)[x_1,x_2, \dots,x_\ell]$. Hence, these polynomials $R_\la$ can be uniquely and explicitly
expanded in terms of the monomials, thus defining an infinite $\ell$-dimensional matrix, which can be explicitly inverted. The 
coefficients appearing in this inversion yield nice combinatorial expressions for our averages. Moreover, these coefficients appear to be related to a well-known algebraic combinatorial context, e.g., the problem of counting the number of subgroups of a finite abelian $p$-group with a given structure. This problem is actually at the heart of the construction of the Hall algebra, and the Hall-Littlewood symmetric polynomials. Furthermore, from the probabilistic 
model and the averages we computed, we obtain combinatorial identities involving $q$-series and the theory of Hall-Littlewood
functions. A precise combinatorial analysis of the previously mentioned coefficients appearing in the inversion process allows us to prove that the Poonen and Rains model and the heuristic model for Tate-Shafarevich groups are compatible.  \medskip

This article is organized as follows: in the next section, we recall the definitions of averages we consider. 
In Section~\ref{section_partitions}, we fix some combinatorial notations and 
recall basic facts about integer partitions; we will also study some combinatorial identities related to the theory of Hall-Littlewood 
symmetric functions. Then, we define in Section~\ref{polynomials} the polynomials $R_\lambda$ and prove the above mentioned 
inversion formula. We also give some properties satisfied by the coefficients appearing in the inversion process, and we show how they are related to the problem of counting subgroups of finite abelian $p$-groups. In Section~\ref{links}, we highlight the links between $R_\lambda$ and finite abelian $p$-groups and deduce some averages. In Section~\ref{philo}, we briefly recall the philosophy of the heuristics on class groups and on Tate-Shafarevich groups and give several consequences of our formulas on heuristics. In particular, concerning the Tate-Shafarevich groups of elliptic curves, we discuss how the heuristics principle and the Poonen and Rains model complement each other. Finally, we end the paper in Section~\ref{combinatorial_consequence} by some combinatorial consequences and related questions.

\section{$u$-averages on finite abelian $p$-groups and motivations}

Throughout the paper, $p$ will always denote a prime number. The letter $H$ (or $H_\lambda$ if $H$ is indexed by an integer 
partition $\lambda$) is used for an isomorphism class of a finite abelian $p$-group and the letter $G$  (or $G_\lambda$)  is used 
for an isomorphism class of a (finite abelian) $p$-group of type $S$. As in \cite{delaunay1, delaunay2}, we say that a $p$-group 
$G$ is of type $S$ if $G$ is a finite abelian $p$-group endowed with a non-degenerate, bilinear, alternating pairing:
 $$
 \beta \; :  \; G \times G \rightarrow \Q/\Z .
 $$ 
An isomorphism of groups of type $S$ is an isomorphism of groups preserving the pairings. A $p$-group of type $S$ has the form 
$G = H \times H$. Reciprocally, each group $G$, of the form $H \times H$ has a unique structure of group of type $S$ up to 
isomorphism. In particular, if $G$ is a group of type $S$ then $|\aut^s(G)|$, the number of automorphisms of $G$ that respect the 
(underlying) pairing~$\beta$, does not depend on $\beta$. Due to these facts, the pairings associated to groups of type $S$ will 
never be specified in the sequel. 

\subsection{Definitions of $u$-averages}
Let $h$ be a complex-valued function defined on isomorphism classes of finite abelian $p$-groups and let $u$ be a nonnegative real number. In \cite{cohen-lenstra},  Cohen and Lenstra defined an $u$-average\footnote{Actually, the $u$-averages $\m(h)$, as well as $\ms(h)$, are defined for all finite abelian groups, not only for $p$-groups. However, for the functions $h$ we have in mind, all the $p$-parts behave independently.} $\m(h)$, of the function $h$ (in particular, if $h$ is the characteristic function of a property ${\mathcal P}$, $\m(h)$ is called the $u$-probability of ${\mathcal P}$), and we have
$$
\m(h) = \frac{\displaystyle \sum_{n\geq 1} p^{-nu}\sum_{H(p^n)} \frac{h(H)}{|\aut(H)|}}{\displaystyle \prod_{j\geq1} \left(1-p^{-u-j}\right)^{-1}},
$$
if the first sum on the numerator converges, and where $\sum_{H(p^n)}$ means that the sum is over all isomorphism classes of finite abelian groups $H$ of order $p^n$. The number $\m(h)$ is indeed an average since whenever $h\equiv 1$ we have $\m(h)=1$, i.e. (see \cite{hall,cohen-lenstra})
$$
 \prod_{j\geq1} \left(1-p^{-u-j}\right)^{-1} = \sum_{n\geq 1} p^{-nu} \sum_{H(p^n)} \frac{1}{|\aut(H)|}.
$$
Cohen and Lenstra introduced such averages as a heuristic model in order to give precise predictions for the behavior of class 
groups of number fields varying in certain natural families. We will come back to this notion and give some examples later.

The aim of \cite{delaunay1} was to adapt Cohen-Lenstra's model, to give predictions for the behavior of Tate-Shafarevich groups 
of elliptic curves varying in certain natural families. Recall that if they are finite (as it is classically conjectured), Tate-Shafarevich 
groups are groups of type $S$. One can define, for a nonnegative real number $u$, the $u$-average (in the sense of groups of 
type $S$), denoted by $\ms(g)$, for a function $g$ defined on isomorphism classes of groups of type $S$. We have
$$
\ms(g) = \frac{\displaystyle \sum_{n\geq 1} p^{-nu} \sum_{G(p^n)} \frac{|G|\, g(G)}{|\auts(G)|}}{\displaystyle \prod_{j\geq1} \left(1-p^{-2u-2j+1}\right)^{-1}},
$$
if the first sum on the numerator converges, and where $\sum_{G(p^n)}$ means that the sum is over all isomorphism classes of 
groups $G$ of type $S$ and order $p^n$ (the sum being empty if $n$ is odd). As above, the number $\ms(g)$ is an average 
since if $g\equiv 1$, then $\ms(g)=1$, i.e. (see \cite{delaunay1}),
$$
\prod_{j\geq1} \left(1-p^{-2u-2j+1}\right)^{-1}= \sum_{n\geq 1} p^{-nu} \sum_{G(p^n)} \frac{|G|}{|\auts(G)|}=\sum_{n\geq 1} p^{-2nu} \sum_{G(p^{2n})} \frac{|G|}{|\auts(G)|}.
$$
As for finite abelian groups, $\ms(g)$ will be called the $u$-probability of ${\mathcal P}$ if $g$ is the characteristic function of a 
property ${\mathcal P}$ for groups of type $S$. It will always be clear in the context whether we are concerned with averages in 
the sense of finite abelian groups or groups of type $S$.
~\\
If $h$ is defined over isomorphism classes of finite abelian $p$-groups, it is not difficult to deduce the $u$-average $\ms(h)$ of 
$h$ restricted to isomorphism classes of $p$-groups of type $S$ from $\m(h)$, the $u$-average of $h$ in the sense of finite 
abelian groups (see \cite{delaunay2, delaunay3}), and reciprocally. 

\subsection{Motivations}
The goal of \cite{delaunay3} was to study the probability laws of  the $p^\ell$-ranks of abelian $p$-groups (resp. $p$-groups of 
type $S$) in the sense of the above probability model. In particular, a closed formula was given for the  $u$-averages of the function 
$h(H)=|H[p^\ell]|$ (resp. $g(G)=|G[p^\ell]|$) for any integer $\ell \geq 0$, where $H[p^\ell]$ denotes the $p^\ell$-torsion elements subgroup of $H$. Such  computations were motivated by a question of M. Bhargava about the heuristics for 
Tate-Shafarevich groups and their comparison with theoretical results of Bhargava and Shankar 
(\cite{bhargava_shankar_1,bhargava_shankar_2}) about the distribution of  $p$-Selmer groups (with $p=2$ and $3$) of all 
elliptic curves. \medskip

From this point of view, our work aims to generalize the study in \cite{delaunay3}, and to give formulas for the $u$-averages of 
$h(H)=|H[p^\ell]|^{m_\ell}$ (resp. $g(G)=|G[p^\ell]|^{m_\ell}$) for all integers $\ell \geq 0$ and $m_\ell \geq 0$. As we will see later, 
the computations are much more technical as they involve  deeper combinatorial tools. Roughly speaking, in order to evaluate the 
$u$-average of $h(H)=|H[p^\ell]|^{m_\ell}$, we will more generally compute the $u$-averages of any function 
$h(H)=|H[p]|^{m_1} |H[p^2]|^{m_2}\cdots |H[p^\ell]|^{m_\ell}$ for all nonnegative integers  
$m_1, m_2, \dots, m_{\ell}$. We  are able to find a formula by induction; but actually we directly obtain a closed formula by using a $\ell$-dimensional matrix inversion result which takes the place of the recursion principle. 
The motivation for this computation will be explained in the next section and comes originally from the following question of 
Poonen: take $g(G)=|G[p^\ell]|^{m_\ell}$; is  it true that $\mszero(g)=\msun(g)\cdot p^{\ell m_\ell}$? We will actually positively answer this question, by proving
the following more general result.
\begin{theorem}\label{coherent} For any positive integer $\ell$, let $m_1,m_2,\dots,m_\ell$ be $\ell$ nonnegative integers. Consider the following function $g$ defined on isomorphism classes of $p$-groups of type $S$:
$$
g(G) = |G[p]|^{m_1}|G[p^2]|^{m_2} \cdots |G[p^\ell]|^{m_\ell}.
$$
Then we have 
$$
\mszero(g) = \msun(g)\,p^{m_1+2m_2+\cdots+ \ell m_\ell}.
$$
\end{theorem}
As we will explain, the fact that the answer to Poonen's question is indeed positive makes the Poonen-Rains model for Selmer groups compatible 
with the heuristics for Tate-Shafarevich groups. \medskip
~\\
 From a combinatorial point of view, our computations and the tools they require (including the probabilistic model for abelian $p$-groups), allow us to obtain a family of interesting combinatorial identities. Some of them are related to summation formulas for Hall-Littlewood symmetric functions, and can be proved in a non trivial way through this theory, while others  seem to be new, or at least not direct consequences of previous known identities of the same type. At the heart of our computations appears some coefficients, denoted by $C_{\la,\mu}(q)$ (where $\la$ and $\mu$ are integer partitions), which satisfy nice combinatorial properties, and, setting $q=1/p$, are related to the number of subgroups $H_\mu$ in a finite abelian $p$-group $H_\la$ (a group $H_\la$ is often called a \emph{group of type $\la$}\footnote{If $\lambda:=(\lambda_1 \geq \lambda_2 \cdots)$ is a partition, a $p$-group $H_\lambda$ is said to be a group of type $\lambda$ if it is isomorphic to $\Z/p^{\lambda_1}\Z \oplus \Z/p^{\lambda_2}\Z \oplus \cdots$. We will not use this terminology in order to avoid any confusion with the notion of group of type $S$ we introduced before.}). In particular, their symmetry relation given by Theorem~\ref{miroir} yields a proof of Theorem~\ref{coherent}, and therefore a positive answer to Poonen's question. As we will see, Theorem~\ref{miroir} is actually a consequence of a well-known property in the framework of finite abelian $p$-groups \cite{Butler, Delsarte}, and therefore related to Hall-Littlewood symmetric polynomials. We shall however give an alternative  proof of this symmetry property by using a general summation formula (due to Lascoux \cite{Lascoux}) of Hall-Littlewood functions. Finally, to prove Theorem~\ref{coherent}, we will also need an evaluation of the number of injective homomorphisms of a group $H_\la$ into the finite abelian $p$-group $H$, associated to a result of Cohen and Lenstra from \cite{cohen-lenstra}.  


\section{Integer partitions and Hall-Littlewood functions}\label{section_partitions}

\subsection{Notations and definitions}
Recall that a \emph{partition} $\la:=(\la_1\geq\la_2\geq\cdots)$ of a nonnegative integer $n$ is a finite decreasing sequence of nonnegative integers whose sum is equal to $n$. We often use the notation $|\la|:=\sum_{i=1}^\ell\la_i=n$, and the integer $\ell=\ell(\la)=max\{i|\la_i>0\}$ is called the \emph{length} of the integer partition $\la$. We will also denote by $\la\vdash n$ the fact that $\la$ is a partition of the integer $n$. The integers $\la_i$ (for $1\leq i\leq \ell$) are the \emph{parts} of $\la$. It is also convenient to write an integer partition by its \emph{multiplicities}: we denote by $m_i=m_i(\la):=\#\{j|\la_j=i\}\geq0$ the multiplicity of the integer $i$ in $\la$. Therefore the partition $\la$ can be written $\displaystyle\la=1^{m_1}2^{m_2}\dots$, and we have $n=m_1+2m_2+\cdots$. Finally, we denote by $\la'$ the \emph{conjugate} partition of $\la$, whose parts are $\la_i':=\#\{j|\la_j\geq i\}$, for $1\leq i\leq \la_1$. Thus we have $|\la'|=|\la|$, $\ell(\la')=\la_1$, and for any integer $i$, $m_i(\la)=\la'_i-\la'_{i+1}$ (and equivalently $m_i(\la')=\la_i-\la_{i+1}$). \\ We will also use the following standard statistic on integer partitions
$$
n(\lambda):=\sum_{i\geq 1}(i-1)\lambda_i=\sum_{i\geq 1}\bi{\lambda'_i}{2}.
$$

\medskip

Next we recall some standard definitions and notations for $q$-series, which can be found in \cite{GasperRahman}. 
Let $q$ be a fixed complex parameter (the ``base'') with $|q|\neq 1$. We define for any real number $a$ and any $k\in\Z$ the \emph{$q$-shifted factorial} by
\begin{equation*}
(a)_k\equiv
(a;q)_k:=\left\{\begin{array}{l}1\;\;\mbox{if}\;\;k=0\\ (1-a)\dots
(1-aq^{k-1})\;\;\mbox{if}\;\;k>0\\1/(1-aq^{-1})\dots
(1-aq^{k})\;\;\mbox{if}\;\;k<0,
\end{array}\right.
\end{equation*}
and $(a)_\infty:=\lim_{k\to+\infty}(a)_k$. The base $q$ can be omitted when there is no confusion (writing $(a)_k$ for $(a;q)_k$, etc). For the sake of simplicity, write for $k\in\Z\cup\{\infty\}$
\begin{equation*}
(a_1,\ldots,a_m)_k:=(a_1)_k\times\cdots\times(a_m)_k.
\end{equation*}
\medskip
Recall also the definition of the \emph{$q$-binomial coefficient}
$$\qbi{n}{k}{q}:=\frac{(q)_n}{(q)_k(q)_{n-k}},$$
which is a polynomial in the variable $q$, with positive integral coefficients. \\

We will need later a special case of the so-called finite $q$-binomial theorem (see for instance \cite{GasperRahman}), which for any complex number $z$ and any nonnegative integer $n$, can be written as follows
\begin{equation}\label{qbin}
\sum_{k=0}^n(-1)^kz^kq^{k(k-1)/2}\qbi{n}{k}{q}=(z)_n.
\end{equation}

Next, we recall some basic facts about the theory of Hall-Littlewood functions, which can be found with more details in \cite[Chapter III]{Macdonald}. Hall-Littlewood polynomials were introduced by Hall to evaluate the number of subgroups of type $\mu$ and cotype $\nu$ in a finite abelian $p$-group of type $\la$, and expressed explicitly later by Littlewood (see for instance \cite{Butler} for detailed explanations of this construction). Let $S_n$ be the symmetric group, $\Lambda_n=\mathbb{Z}[x_1,\dots,x_n]^{S_n}$ be the ring of symmetric polynomials in $n$ independant variables and $\Lambda$ be the ring of symmetric functions in countably many independant variables. For a set of $n$ variables $x=\{x_1,\dots,x_n\}$, the Hall-Littlewood
polynomials $P_{\la}(x;q)$ are defined by
$$
P_{\la}(x;q):= \prod_{i\geq
1}\frac{(1-q)^{m_i}}{(q)_{m_i}}\,\sum_{w\in
S_n}w\left(x_1^{\la_1}\ldots x_n^{\la_n}
\prod_{i<j}\frac{x_i-qx_j}{ x_i-x_j}\right),
$$
where $w$ acts by permuting the variables. These polynomials are symmetric in $x$, homogeneous of degree $|\la|$, with coefficients in $\mathbb{Z}[q]$, and form a $\mathbb{Z}[q]$ basis of $\Lambda_n[q]$, which interpolates between monomial (for $q=1$) and Schur (for $q=0$) polynomials. Moreover, thanks to their classical stability property, they may be extended to the Hall-Littlewood functions in an infinite number of variables in order to form a $\mathbb{Z}[q]$ basis of $\Lambda[q]$. Another useful result is the specialization
\begin{equation}\label{principale}
P_{\la}(z,zq,\dots,zq^{n-1};q)=\frac{z^{|\lambda|}q^{n(\lambda)}(q)_n}{(q)_{n-\ell(\lambda)}b_\la(q)},
\end{equation}
where 
$$b_\la(q):=\prod_{i\geq1}(q)_{\lambda'_i-\lambda'_{i+1}}.$$

Finally, by using the Cauchy identity for Hall-Littlewood polynomials \cite[p.~224, (4.4)]{Macdonald}, the specialization~\eqref{principale}, and the Pieri formula \cite[p.~224, (3.2)]{Macdonald}, it is possible to derive the following identity from \cite{Macdonald}, which can be seen as a $q$-binomial theorem for Hall-Littlewood functions
\begin{equation}\label{qbinhl}
\sum_{\la}q^{n(\la)}(a;q^{-1})_{\ell(\la)}P_\la(x;q)=\prod_{i\geq1}\frac{1-ax_i}{1-x_i},
\end{equation}
where $|q|<1$, $a$ is a complex number, and the summation is taken over all integer partitions.
\subsection{Lascoux's summation of Hall-Littlewood functions}
The Rogers-Rama-nujan identities are two of the most famous $q$-series identities, with deep connections to many branches of mathematics. Lots of proofs of these identities have been given, among which one by Stembridge in \cite{stembridge}, who generalized work by Macdonald from \cite{Macdonald} to show that the Rogers-Ramanujan identities (among others) may be obtained by appropriate specializations of finite summations for Hall-Littlewood polynomials. In \cite{ASW}, Andrews, Schilling and Warnaar generalized the Rogers-Ramanujan identities to three identities labelled by the Lie algebra $A_2$. Later, seeking a way of proving the latter through the theory of Hall-Littlewood polynomials as Stembridge did for the Rogers-Ramanujan identities, Warnaar showed in \cite{Warnaar} the following $A_2$-type identity for Hall-Littlewood functions
\begin{equation}\label{warnaara2hl}
\sum_{\la,\,\mu}q^{n(\la)+n(\mu)-(\la'|\mu')}P_\la(x;q)P_\mu(y;q)
=\prod_{i\geq1}\frac{1}{(1-x_i)(1-y_i)}\prod_{i,j\geq1}\frac{1-x_iy_j}{1-x_iy_j/q}\,,
\end{equation}
for $|q|<1$, where the sum is over all pairs of integers partitions, and $(\la|\mu):=\sum_{i\geq1}\lambda_i\mu_i$. In \cite{Warnaar}, Warnaar explains how \eqref{warnaara2hl} yields one of the three $A_2$-type Rogers-Ramanujan identities from \cite{ASW}.\\
For our purpose, we will actually need a generalization of \eqref{warnaara2hl}, which was discovered and proved by Lascoux in \cite{Lascoux}. To this aim, he used the modified Hall-Littlewood polynomials $Q'_\mu$, which form the adjoint basis of  $P_\la$ with respect to the Hall scalar product (see \cite{Lascoux, Macdonald} for details). More precisely, Lascoux's formula can be written as
\begin{equation}\label{lascoux}
\sum_{\la}\sum_{\mu\subseteq\la}P_\la(x;q)P_\mu(y;q)b_\la(q)Q'_{\la/\mu}(\bold{1};q)
=\prod_{i\geq1}\frac{1}{1-x_i}\prod_{j\geq1}\frac{1-qx_iy_j}{1-x_iy_j}\,, \;(|q|<1)\,,
\end{equation}
where $\mu\subseteq\la$ means that $\mu_i\leq\la_i$ for all integers $i$, $Q'_{\la/\mu}$ are the skew modified Hall-Littlewood polynomials, and the coefficients $Q'_{\la/\mu}(\bold{1};q)$ are explicitely given by
\begin{equation}\label{coeflascoux}
Q'_{\la/\mu}(\bold{1};q)=q^{|\mu|+n(\la)+n(\mu)-(\la'|\mu')}\prod_{i\geq1}\qbi{\la'_i-\mu'_{i+1}}{\la'_i-\mu'_i}{q}.
\end{equation}
It is easy to see that $Q'_{\la/\mu}(\bold{1};q)$ are polynomials in the variable $q$, with positive integral coefficients, for which Lascoux gave a combinatorial interpretation in \cite{Lascoux}. As remarked in \cite{WarnaarZudilin}, setting $q=1/p$ with $p$ a prime number, $Q'_{\la/\mu}(\bold{1};1/p)$ are related to the number of subgroups of type $\mu$ in a finite abelian $p$-group of type $\la$. We will come back to this link in the next section.
Finally, although it is not trivial at first sight, Lascoux showed in \cite{Lascoux} how Warnaar's identity \eqref{warnaara2hl} is a consequence of \eqref{lascoux}.


\section{The polynomials $R_{\lambda} (x,t)$}\label{polynomials}
For our purpose, we fix a positive integer $\ell$, and we consider a partition $\la=(\la_1\geq\la_2\geq\cdots)$ such that  $\la_1\leq \ell$ (i.e., the length of the conjugate partition $\la'$ satisfies $\ell(\la')\leq \ell$). For a set of $\ell$ variables $x=\{x_1,\dots,x_\ell\}$ (adding by  convention $x_0:=1$) and a complex parameter $t$, we define the polynomial\footnote{Note that we take the parameter $t$ in order to avoid confusions since we will need to specialize $t=q$ and $t=1/q$ later.}
\begin{equation}\label{polyn}
R_{\la}(x;t):=\prod_{i=1}^\ell\,\prod_{j=\la_{i+1}'}^{\la_i'-1}(x_i-t^jx_{i-1}),
\end{equation}
which can be rewritten through the multiplicities $m_1,\dots,m_\ell$ of the partition $\la$ as follows
\begin{eqnarray}
R_{\la}(x;t)&=&\prod_{i=1}^\ell x_i^{\la_i'-\la_{i+1}'}\prod_{j=0}^{\la_i'-\la_{i+1}'-1}\left(1-t^{j+\la_{i+1}'}\frac{x_{i-1}}{x_i}\right)\nonumber\\
&=&\prod_{i=1}^\ell x_i^{m_i}(t^{m_{i+1}+\dots+m_\ell}x_{i-1}/x_i;t)_{m_i}\label{polyn0}.
\end{eqnarray}
It is clear that whenever $\lambda$ runs through all partitions with $\lambda_1\leq \ell$, then the polynomials $R_\lambda(x;t)$ form a basis of $\Z(t)[x_1,x_2,\dots, x_\ell]$.  
~\\
Our goal is to expand the polynomial $R_\la$ in terms of the monomials $x_1^{k_1}\dots x_\ell^{k_\ell}$, and to invert this relation. To this aim, writing $\la=1^{m_1}2^{m_2}\dots$, we have to express the monomial $x_1^{m_1}\dots x_\ell^{m_\ell}$ in terms of the polynomials $R_{\mu}(x;t)$, where $\mu=(\mu_1\geq\mu_2\geq\dots)$ are integer partitions included in $\la$, i.e., satisfying $\mu_i\leq\la_i$ for any integer $i$ (note that this is equivalent to the conditions $\mu'_i\leq\la'_i$ for any integer $i$).\medskip
~\\
It is easy to see that $R_\lambda$ can be expanded as a sum of monomials $x^\mu$ for $\mu \subseteq \lambda$, and this is what we explicitly do in the next proposition.
\begin{proposition} We have 
\begin{multline*}
R_{\la}(x;q)=\sum_{\mu\subseteq\la}x_1^{\mu'_1-\mu'_2}\dots x_\ell^{\mu'_\ell}\,(-1)^{|\la|-|\mu|}\,q^{n(\la\setminus\mu)}\\
\times q^{\la'_2(\la'_1-\mu'_1)+\dots+\la'_\ell(\la'_{\ell-1}-\mu'_{\ell-1})}\prod_{i\geq1}\qbi{\la'_i-\la'_{i+1}}{\la'_i-\mu'_i}{q}.
\end{multline*}
\end{proposition}
\begin{proof}
We expand the expression \eqref{polyn0} of the polynomials $R_\la(x;q)$ by using $\ell$ times the $q$-binomial theorem \eqref{qbin}, which yields
\begin{multline*}
R_{\la}(x;q)=\sum_{k_1,\dots,k_\ell}x_1^{m_1-k_1+k_2}\dots x_{\ell-1}^{m_{\ell-1}-k_{\ell-1}+k_\ell}x_\ell^{m_\ell-k_\ell}\,(-1)^{k_1+\dots+k_\ell}\\
\times q^{\bi{k_1}{2}+\dots+\bi{k_\ell}{2}}q^{(m_2+\dots+m_\ell)k_1+\dots+m_\ell k_{\ell-1}}\qbi{m_1}{k_1}{q}\dots\qbi{m_\ell}{k_\ell}{q}.
\end{multline*}
Definig the integer partition $\mu$ by $\mu=1^{k_1}2^{k_2}\dots \ell^{k_\ell}$, and replacing in this expression $k_i$ by $\la'_i-\mu'_i$, we obtain the formula.
\end{proof}
~\\
We are now interested in the reciprocal identity, i.e., the expansion of the monomials $x^\lambda$ in terms of the polynomials $R_\mu$.  For this, we need  the following multidimensional matrix inversion, which was communicated to us by Michael Schlosser. 
\begin{lemma}\label{lemme_un}
The infinite lower-triangular $\ell$-dimensional matrices $F=(F_{\textbf{a},\textbf{b}})_{\textbf{a},\textbf{b}\,\in\mathbb{Z}^\ell}$ and $G=(G_{\textbf{a},\textbf{b}})_{\textbf{a},\textbf{b}\,\in\mathbb{Z}^\ell}$ are inverses of each other, where
$$F_{\textbf{a},\textbf{b}}=(-1)^{|\textbf{a}|-|\textbf{b}|}q^{\sum_{i=1}^\ell\left(\bi{a_i+1}{2}+\bi{b_i}{2}+a_ib_{i+1}\right)}\prod_{i=1}^\ell\frac{(q)_{a_i-b_{i+1}}}{(q)_{a_i-b_{i}}}\,,$$
and 
$$G_{\textbf{a},\textbf{b}}=q^{-\sum_{i=1}^\ell b_i(a_i+a_{i+1})}\prod_{i=1}^\ell\frac{1}{(q)_{a_i-b_{i}}(q)_{b_i-a_{i+1}}}\,,$$
with $a_{\ell+1}=b_{\ell+1}=0$ and $|\textbf{a}|=a_1+\dots+a_\ell$, $|\textbf{b}|=b_1+\dots+b_\ell$.
\end{lemma}
\begin{proof}
We will see that the previous inversion follows from a suitable specialization of a multidimensional matrix inversion due to Schlosser in \cite{Schlosser}, which is itself a special case of the inversion of the Pieri rule for Macdonald polynomials \cite[Theorem~2.6]{LassalleSchlosser}. First recall that the infinite lower-triangular $\ell$-dimensional matrices $F=(F_{\textbf{a},\textbf{b}})_{\textbf{a},\textbf{b}\,\in\mathbb{Z}^\ell}$ and $G=(G_{\textbf{a},\textbf{b}})_{\textbf{a},\textbf{b}\,\in\mathbb{Z}^\ell}$ are inverses of each other means that for any pair of infinite $\ell$-dimensional vectors $\alpha_{\textbf{a}}$ and $\beta_{\textbf{b}}$, we have
\begin{equation}\label{definversion}
\alpha_{\textbf{a}}=\sum_{\textbf{b}\leq\textbf{a}}F_{\textbf{a},\textbf{b}}\,\beta_{\textbf{b}}\Leftrightarrow\beta_{\textbf{a}}=\sum_{\textbf{b}\leq\textbf{a}}G_{\textbf{a},\textbf{b}}\,\alpha_{\textbf{b}},
\end{equation}
where $\textbf{b}\leq\textbf{a}$ means that $b_i\leq a_i$ for all integers $i\in\{1,\dots,\ell\}$. Now in \cite[(A.9)]{Schlosser}, replace $n$ by $\ell$, let $t_0\to 0$ and $t_i\mapsto u_i/u_{i+1}$ for $i=1,\dots,\ell$. We denote by $(f_{\textbf{a},\textbf{b}})_{\textbf{a},\textbf{b}\,\in\mathbb{Z}^\ell}$ and $(g_{\textbf{a},\textbf{b}})_{\textbf{a},\textbf{b}\,\in\mathbb{Z}^\ell}$ the resulting infinite lower-triangular $\ell$-dimensional matrices, which are inverses of each other. Then we define infinite invertible diagonal $\ell$-dimensional matrices $D$ and $E$ by
\begin{multline*}
D_{\textbf{a}}:=u_1^{a_1}u_2^{2a_2-a_1}\dots u_\ell^{\ell a_\ell-(\ell-1)a_{\ell-1}}u_{\ell+1}^{-\ell a_\ell}(-1)^{a_\ell}\prod_{i=1}^{\ell-1}(q)_{a_i-a_{i+1}}\\
\times q^{a_l^2/2+a_\ell/2-\sum_{i=1}^\ell(a_ia_{i+1}-ia_i)}
\end{multline*}
and
\begin{multline*}
E_{\textbf{b}}:=u_1^{b_1}u_2^{2b_2-b_1}\dots u_\ell^{\ell b_\ell-(\ell-1)b_{\ell-1}}u_{\ell+1}^{-\ell b_\ell}(-1)^{b_\ell}(q)_{b_\ell}\times q^{b_\ell^2/2+b_\ell/2-\sum_{i=1}^\ell(b_i^2-ib_i)}.
\end{multline*}
It is clear that through the definition in \eqref{definversion}, the matrices  $(D_{\textbf{b}}f_{\textbf{a},\textbf{b}}E_{\textbf{a}}^{-1})_{\textbf{a},\textbf{b}\,\in\mathbb{Z}^\ell}$ and $(E_{\textbf{b}}g_{\textbf{a},\textbf{b}}D_{\textbf{a}}^{-1})_{\textbf{a},\textbf{b}\,\in\mathbb{Z}^\ell}$ are still inverses of each other. Finally, we let $u_1\to\infty$, then $u_2\to\infty$, ... , $u_\ell\to\infty$ (in this order) in the previous matrices, which yields after some complicated but elementary manipulations
$$\lim_{u_1,\dots,u_\ell\to\infty}D_{\textbf{b}}f_{\textbf{a},\textbf{b}}E_{\textbf{a}}^{-1}=F_{\textbf{a},\textbf{b}},$$
and
$$\lim_{u_1,\dots,u_\ell\to\infty}E_{\textbf{b}}g_{\textbf{a},\textbf{b}}D_{\textbf{a}}^{-1}=G_{\textbf{a},\textbf{b}}.$$
This completes the proof.
\end{proof}
Now we are able to prove the following result, which is our desired inversion.
\begin{theorem}\label{inversion}
For any positive integer $\ell$, and an integer partition $\la=1^{m_1}2^{m_2}\dots \ell^{m_\ell}$, we have
\begin{equation*}\label{inversion}
x_1^{m_1}\dots x_\ell^{m_\ell}=\sum_{\mu\subseteq\la}C_{\la,\,\mu}(q)\,R_\mu(x;q),
\end{equation*}
where 
\begin{equation}\label{coeff}
C_{\la,\,\mu}(q):=q^{\sum_{i=1}^\ell\mu'_{i+1}(\la'_i-\mu'_i)}
\,\prod_{i\geq1}\qbi{\la'_i-\mu'_{i+1}}{\la'_i-\mu'_i}{q}.
\end{equation}
\end{theorem}
\begin{proof}
Recall that we have 
\begin{multline*}
R_{\la}(x;q)=\sum_{\mu\subseteq\la}x_1^{\mu'_1-\mu'_2}\dots x_\ell^{\mu'_\ell}\,(-1)^{|\la|-|\mu|}\,q^{\bi{\la'_1-\mu'_1}{2}+\dots+\bi{\la'_\ell-\mu'_\ell}{2}}\\
\times q^{\la'_2(\la'_1-\mu'_1)+\dots+\la'_\ell(\la'_{\ell-1}-\mu'_{\ell-1})}\qbi{\la'_1-\la'_2}{\la'_1-\mu'_1}{q}\dots\qbi{\la'_\ell}{\la'_\ell-\mu'_\ell}{q}.
\end{multline*}
In order to use the inversion of Lemma~\ref{lemme_un}, we observe that this relation is equivalent to the following:
\begin{equation}\label{applinv}
\beta_{\textbf{a}}=\sum_{\textbf{b}\leq\textbf{a}}G_{\textbf{a},\textbf{b}}\,\alpha_{\textbf{b}},
\end{equation}
where $a_i=\la'_i$, $b_i=\mu'_i$,  $$\beta_{\textbf{a}}:=(-1)^{|\textbf{a}|}\frac{q^{-\sum_{i=1}^\ell\left(\bi{a_i}{2}+a_ia_{i+1}\right)}}{(q)_{a_1-a_2}\dots(q)_{a_l}}R_\la(x;q),$$
and
$$\alpha_{\textbf{b}}:=(-1)^{|\textbf{b}|}q^{\sum_{i=1}^\ell\bi{b_i+1}{2}}x_1^{b_1-b_2}\dots x_\ell^{b_\ell}.$$
The multidimensional matrix inversion in Lemma~\ref{lemme_un} applied to \eqref{applinv} yields
$$
\alpha_{\textbf{a}}=\sum_{\textbf{b}\leq\textbf{a}}F_{\textbf{a},\textbf{b}}\,\beta_{\textbf{b}},
$$
which is equivalent to
$$
x_1^{a_1-a_2}\dots x_\ell^{a_\ell}=\sum_{\textbf{b}\leq\textbf{a}}q^{\sum_{i=1}^\ell b_{i+1}(a_i-b_i)}\prod_{i=1}^\ell\frac{(q)_{a_i-b_{i+1}}}{(q)_{b_i-b_{i+1}}(q)_{a_i-b_{i}}}R_\mu(x;q),
$$
and this completes the proof after replacing the variables $a_i$ by $\la'_i$ and $b_i$ by $\mu'_i$, and then recalling the definition of the $q$-binomial coefficient.
\end{proof}
~\\
{\bf Remark.}
It is easy to see from its definition that $C_{\la,\,\mu}(q)$ is a polynomial in the variable $q$, with positive integral coefficients. Moreover, thanks to the classical transformation $\qbi{n}{k}{1/q}=q^{k(k-n)}\qbi{n}{k}{q}$, it is related to the coefficient $Q'_{\la/\mu}(\bold{1};q)$ (defined in \eqref{coeflascoux}) by 
\begin{equation}\label{relation}
C_{\la,\,\mu}(1/q)=q^{n(\mu)-n(\la)}Q'_{\la/\mu}(\bold{1};q),
\end{equation}
which, by setting $q=1/p$ with $p$ a prime number, counts exactly the number of subgroups of type $\mu$ in a finite abelian $p$-group of type $\la$ (see \cite{Delsarte, Butler}).~\\
~\\
Thanks to this remark, the following symmetry property satisfied by the coefficients $C_{\la,\,\mu}(q)$ becomes a direct consequence of the Pontryagin duality for finite abelian groups. However, we give below a very natural proof which uses  Lascoux's formula \eqref{lascoux}, therefore highlighting and exploiting the link, given by \eqref{relation}, between these coefficients and Hall-Littlewood functions.
\begin{theorem}\label{miroir}
Let $m$ and $k$ be two nonnegative integers, and $\la$ be a partition of $m$. Then we have
\begin{equation*}
\sum_{|\mu|=k\atop\mu\subseteq\la}C_{\la,\,\mu}(q)=\sum_{|\mu|=m-k\atop\mu\subseteq\la}C_{\la,\,\mu}(q).
\end{equation*}
Equivalently, the polynomial $\displaystyle\sum_{\mu \subseteq \lambda} C_{\lambda,\mu}(q) T^{|\mu|} \in \N[q][T]
$ is self-reciprocal.
\end{theorem}
\begin{proof}
We will actually show the following equivalent form of the theorem
\begin{equation}\label{conj'}
\sum_{|\mu|=k\atop\mu\subseteq\la}C_{\la,\,\mu}(1/q)=\sum_{|\mu|=m-k\atop\mu\subseteq\la}C_{\la,\,\mu}(1/q).
\end{equation}
In Lascoux's formula \eqref{lascoux}, consider the specialization $y=\{z,zq,zq^2,\dots\}$, and use the $n\to\infty$ case of \eqref{principale}, together with \eqref{relation} and the homogeneity of the Hall-Littlewood polynomials to derive
\begin{equation}\label{fgcla,mu}
\sum_{\la}\sum_{\mu\subseteq\la}q^{n(\la)}P_\la(x;q)z^{|\mu|}C_{\lambda,\mu}(1/q)=\prod_{i\geq1}\frac{1}{(1-x_i)(1-zx_i)},
\end{equation}
which is the generating function for the left-hand side of \eqref{conj'}.
We then compute the same generating function for the right-hand side of \eqref{conj'}
\begin{eqnarray*}
&&\sum_{k\geq0}\sum_\la q^{n(\la)}P_\la(x;q)z^{|\la|-k}\sum_{|\mu|=|\la|-k\atop\mu \subseteq \lambda} C_{\lambda,\mu}(1/q)\\
 &&\hskip 4cm=\sum_{\la,\mu\,|\,\mu \subseteq \lambda}q^{n(\la)}P_\la(x;q)z^{|\la|-|\mu|}C_{\lambda,\mu}(1/q)\\
&&\hskip 4cm=\sum_{\la,\mu\,|\,\mu \subseteq \lambda}q^{n(\la)}P_\la(zx;q)(1/z)^{|\mu|}C_{\lambda,\mu}(1/q),
\end{eqnarray*}
where we have used the homogeneity of the Hall-Littlewood polynomials and the notation $zx$ for the set of variables $\{zx_i,i\geq1\}$. Therefore we can use  \eqref{fgcla,mu} with $x$ replaced by $zx$ and $z$ by $1/z$, to get
$$\prod_{i\geq1}\frac{1}{(1-zx_i)(1-zx_i/z)},$$
which is obviously equal to the right-hand side of \eqref{fgcla,mu}. We thus have the identity
$$\sum_{\la,\nu\,|\,\nu \subseteq \lambda}q^{n(\la)}P_\la(x;q)z^{|\nu|}C_{\lambda,\nu}(1/q)=\sum_{\la,\nu\,|\,\nu \subseteq \lambda}q^{n(\la)}P_\la(x;q)z^{|\la|-|\nu|}C_{\lambda,\nu}(1/q),
$$
which yields \eqref{conj'}, by using the fact that $P_\la(x;q)$ is a $\mathbb{Z}[q]$-basis of $\Lambda[q]$ and extracting the coefficient of $z^kq^{n(\la)}P_\la(x;q)$ on both sides.
\end{proof}

\medskip 
\noindent
{\bf Examples.}~\\
1- First, we have $x_1=R_{1^1} + 1$, and more generally 
$$
x_n = R_{n^1} + R_{(n-1)^1} + \cdots + R_{1^1} + 1.
$$
That is exactly the relation obtained in \cite{delaunay3} in order to compute the $u$-average of the function $H \mapsto |H[p^n]|$. \smallskip
~\\
2- One can also see that $x_1^2= R_{1^2} + (t+1)R_{1^1} + 1$ and that
$$
x_1x_2=  R_{1^12^1} + tR_{2^1} + R_{1^2} + (t+1) R_{1^1} + 1.
$$
Therefore 
$$
x_2^2 = R_{2^2} + (t+1) R_{1^12^1} + t(t+1) R_{2^1} + R_{1^2} + (t+1) R_{1^1} + 1.
$$
3- As we saw the polynomial 
$$
\sum_{\mu \subseteq \lambda} C_{\lambda,\mu}(t) T^{|\mu|} \in \Z(t)[T]
$$
is in fact a self-reciprocal polynomial. Of course, this can be directly checked in the above two examples. \smallskip
~\\
4- Let us consider the case $x_1^m$. Then the partitions $\mu \subseteq 1^m$ are exactly the partitions~$1^k$ for $k=0, 1, \dots, m$. Hence,
\begin{equation}\label{exemple_poonen}
x_1^m = \sum_{k=0}^m \qbi{m}{k}{t} R_{1^k}.
\end{equation}
We will come back later to this example.\\
~\\
{\bf Remark.}
As Ole Warnaar communicated to us, there is an interesting context containing generalizations of Theorems~4 and 5. Indeed, studying properties for Macdonald interpolation polynomials (which are symmetric and whose top homogeneous layer are the classical symmetric Macdonald polynomials, themselves being two parameters generalizations of the Hall-Littlewood polynomials), Okounkov defined in \cite{Ok} the $(q,t)$-binomial coefficients $\qbi{\la}{\mu}{q,t}$, where $\la$ and $\mu$ are integer partitions and $q$, $t$ are complex parameters. These $(q,t)$-binomial coefficients are defined as quotients of special values of the Macdonald interpolation polynomials, and as remarked by Warnaar, it turns out by using \eqref{relation} that we have
$$C_{\la,\mu}(1/q)=\qbi{\la}{\mu}{0,q}.$$
Many properties for these $(q,t)$-binomial coefficients are known, some of which should therefore have interesting algebraic number theoritic applications in our context. As we only need the Hall-Littlewood case for our purpose here, we plan to study further this connection with Macdonald polynomials in a later work. We only mention the following generalization of Theorem~5 (the latter being obtained by setting $\nu=0$, $q=0$ and then replacing $t$ by $q$), which is proved by Lascoux, Rains and Warnaar in \cite{LRW} (actually in the more general context of nonsymmetric Macdonald interpolation polynomials): 
$$\sum_{|\mu|=k}\qbi{\la}{\mu}{q,t}\qbi{\mu}{\nu}{q,t}=\sum_{|\mu|=|\la|+|\nu|-k}\qbi{\la}{\mu}{q,t}\qbi{\mu}{\nu}{q,t},$$
where $\la$ and $\nu$ are fixed integer partitions. Moreover, still remarked by Warnaar, the inversion  in Theorem~4 can 
also be proved using a ($q\to0$) limit case of the inversion for $(q,t)$-binomial coefficients, which can be found in \cite{Ok}, and can be stated as follows:
$$\sum_{\mu} (-1)^{|\mu|-|\lambda|}q^{n(\mu')-n(\lambda')}t^{n(\lambda)-n(\mu)}
\qbi{\lambda}{\mu}{q,t} \qbi{\mu}{\nu}{1/q,1/t}=\delta_{\la,\nu},$$
where $\delta_{\la,\nu}$ stands for the Kronecker symbol.


\section{Links with finite abelian groups}\label{links}
\subsection{Finite abelian $p$-groups and the polynomial $R_\lambda$}

We have already seen in the previous section that the inversion we needed for our computations of $u$-averages is related to finite abelian $p$-groups, through the coefficients $C_{\la,\mu}$. We highlight here this link. There is a classical bijection between the set of partitions of a fixed nonnegative integer $m$ and the isomorphism classes of finite abelian $p$-groups of order $p^m$. Indeed, to any partition $\lambda= 1^{m_1}2^{m_2}\cdots \ell^{m_\ell}$ of $m$, we associate the finite abelian $p$-group
$$
H_\lambda = \left(\Z/p\Z\right)^{m_1} \oplus \left(\Z/p^2\Z\right)^{m_2} \oplus \cdots \oplus \left(\Z/p^\ell\Z\right)^{m_\ell}.
$$
Recall that $\lambda_1', \lambda_2', \dots,\lambda'_\ell$ denote the parts of the conjugate of $\lambda$ (therefore  $\lambda'_i =\sum_{j\geq i} m_j$), and we have (see \cite{cohen-lenstra})
\begin{equation}\label{*}
|\aut(H_\lambda)| = p^{\lambda'^{2}_{1}+\lambda'^{2}_{2} +\cdots + \lambda'^{2}_{\ell}} \prod_{j=1}^\ell (1/p;1/p)_{m_j}.
\end{equation} 
Furthermore if $G_\lambda \simeq H_\lambda \times H_\lambda$ is a group of type $S$, we have (see \cite{delaunay1}):
$$
|\auts(G_\lambda)| = p^{2(\lambda'^{2}_{1}+\lambda'^{2}_{2} +\cdots + \lambda'^{2}_{\ell}) + m} \prod_{j=1}^\ell  (1/p^2;1/p^2)_{m_j}.
$$
Remark also that we have
$$
H_\lambda[p^k] \simeq (\Z/p\Z)^{m_1} \oplus (\Z/p^2\Z)^{m_2} \oplus \cdots \oplus (\Z/p^{k-1}\Z)^{m_{k-1}} \oplus (\Z/p^k\Z)^{\lambda'_k},
$$
hence 
\begin{equation}\label{**}
|H_\lambda[p^k]|=p^{\lambda_1'+\cdots+\lambda_k'}.
\end{equation}

The following definitions will be useful in the sequel.
\begin{definition}\label{def_1d} For a set of $\ell$ variables $x=\{x_1,\dots,x_\ell\}$, let $T(x)=T(x_1,\dots,x_\ell)$ be a polynomial and $H$ be a finite abelian $p$-group. We define 
$$
T(H)= T(|H[p]|,|H[p^2]|,\dots,|H[p^\ell]|).
$$
\end{definition}
\begin{definition}\label{def_2d} If $\mu= 1^{n_1}2^{n_2}\cdots \ell^{n_\ell}$ is an integer partition, we denote by $x^\mu$ the following monomial 
$$
x^\mu:=x_1^{n_1} x_2^{n_2} \cdots x_\ell^{n_\ell}.
$$
\end{definition}
\medskip
\noindent
Hence, with these definitions, the $u$-average of $|H[p^\ell]|$ can be simply written as $\m(x_\ell)$. More generally, the 
$u$-average of $|H[p]|^{n_1} |H[p^2]|^{n_2} \cdots |H[p^\ell]|^{n_\ell}$ is simply $\m(x^\mu)$ where $\mu$ denotes the integer partition $\mu= 1^{n_1}2^{n_2}\cdots \ell^{n_\ell}$.\medskip
~\\
In order to give a closed formula for the latter, we state the following result. 
\begin{theorem}\label{nbr_inj} For a positive integer $\ell$, let $\lambda=1^{m_1}2^{m_2}\cdots \ell^{m_\ell}$  be an integer partition and let $H$ be a finite abelian $p$-group. Then the number of injective homomorphisms of $H_\lambda$ into $H$ is given by
$$
\prod_{i=1}^\ell \prod_{j=\lambda_{i+1}'}^{\lambda_i'-1} \left( |H[p^i]| - p^j|H[p^{i-1}]|\right),
$$
or in other words,
$$
\big|\left\{ \hominj(H_\lambda,H)\right\}\big| = R_\lambda(H;p).
$$
\end{theorem} 
\begin{proof}  
This is probably a  known formula that can be found in the literature, however we give here a sketch of proof. Write 
$$
H_\lambda = \left(\Z/p\Z\right)^{m_1} \oplus \left(\Z/p^2\Z\right)^{m_2} \oplus \cdots \oplus \left(\Z/p^\ell\Z\right)^{m_\ell},
$$
and take $e_1,e_2,\ldots,e_{\lambda'_1}$ a ``canonical basis" in an obvious sense so that $e_1, \dots, e_{m_1}$ are of order $p$ and $e_{m_1+m_2+\cdots+m_{\ell-1}+1},\dots, e_{\lambda'_1}$ are of order $p^\ell$. In order to define an injective homomorphism of $H_\lambda$ into $H$, we need to specify the
image $f_i$ of each $e_i$. The image $f_ {\lambda'_1}$ of $e_{\lambda'_1}$ is any point of order $p^\ell$ and 
there are $|H[p^\ell]|-|H[p^{\ell-1}]|$ choices for this image (these correspond to the factor $i=\ell$ and $j=0=\lambda'_{\ell+1}$ in the theorem). 
Assume that $k+1\leq \lambda'_1$, and that $f_{\lambda'_1}, \dots, f_{k+1}$ are already defined. Let $p^t$ be the order of 
$e_k$. Note that the order of each $f_{\lambda'_1}, \dots, f_{k+1}$ is $\geq p^t$. Now the image $f_k$ of $e_k$ must be an element 
of order $p^t$ such that, for any $r=1, \dots, p^{t}-1$, $rf_k $ is not in the subgroup $\langle f_{k+1}, \dots, f_{\lambda'_1} \rangle$ generated by the 
images that are already defined. The latter means that the image of $f_k$ in the quotient group 
$\tilde{H}=H/\langle f_{k+1}, \dots, f_{\lambda'_1} \rangle$ has order $p^t$. The number of elements of order $p^t$ in $\tilde{H}$ is
$|\tilde{H}[p^{t}]| - |\tilde{H}[p^{t-1}]|$. For each of these elements there are $|\langle  f_{k+1}, \dots, f_{\lambda'_1} \rangle[p^{t}]|$  corresponding pre-images of order $p^t$ in $H$. Hence the number of possibilities for $f_k$ is 
$$
\left(|\tilde{H}[p^{t}]| - |\tilde{H}[p^{t-1}]\right) |\langle  f_{k+1}, \dots, f_{\lambda'_1} \rangle[p^{t}]| = |H[p^t]|-p^{\lambda_1'-k}|H[p^{t-1}]|.
$$
This is exactly the term occurring with $i=t$ and $j=\lambda'_1-k$ in the formula of the theorem (and we indeed have 
$\lambda'_{t+1} \leq \lambda'_1-k\leq \lambda'_t-1$).
\end{proof}
We derive the following consequence.
\begin{corollary}\label{coro_1r} Let $n$ and $m$ be nonnegative integers, and $\lambda$ be a fixed partition of $m$. Then we have
$$
\sum_{\mu \vdash n} \frac{R_\lambda(H_\mu;p)}{|\aut H_\mu|} = \frac{1}{p^{n-m} (1/p;1/p)_{n-m}},
$$
which can equivalently be written as follows in terms of generating functions
$$
\sum_{n\geq 0} z^n \sum_{\mu \vdash n} \frac{R_\lambda(H_\mu;p)}{|\aut H_\mu|} = \frac{z^m }{(z/p;1/p)_\infty}.
$$
\end{corollary}
\begin{proof}
We clearly have 
$$
|\hominj(H_\lambda,H)|=|\aut(H_\lambda)|\cdot |\{J \mbox{~subgroup\, of~} H, J\simeq H_\lambda \}|.
$$ 
Then Proposition~4.1 of \cite{cohen-lenstra} gives
$$
|\aut(H_\lambda)|  \sum_{H(p^n)} \frac{|\{J \mbox{~subgroup\, of~} H, J\simeq H_\lambda \}|}{|\aut(H)|} = \frac{1}{p^{n-m}(1/p;1/p)_{n-m}}.
$$
We obtain the first equality from Theorem~\ref{nbr_inj} and the correspondence $(\mu \vdash n) \leftrightarrow $ ($H_\mu$ of order $p^n$). The second equality comes from
Euler's identity (see for instance \cite{GasperRahman})
\begin{eqnarray*}
\sum_{n\geq m} \frac{z^n}{p^{n-m}(1/p;1;p)_{n-m}} & = & z^m \sum_{n\geq 0} \frac{z^n}{p^{n}(1/p;1;p)_{n}}\\
&=& z^m  \frac{1}{(z/p;1/p)_\infty}.
\end{eqnarray*}
\end{proof}
\noindent
Using Theorem~\ref{inversion}, Definitions~\ref{def_1d},~\ref{def_2d}, and Corollary~\ref{coro_1r}, we obtain the following result.
\begin{theorem}\label{function_gen} 
Let $\lambda$ be a fixed integer partition, and $z$ a complex number. We have
$$
\sum_{\mu}  \frac{x^\lambda(H_\mu)}{|\aut H_\mu|} \; z^{|\mu|} = \frac{1}{(z/p;1/p)_\infty} \sum_{\nu \subseteq \lambda} C_{\lambda,\nu}(p) \; z^{|\nu|},
$$
where the sum on the left-hand side is over all integer partitions $\mu$, and the sum on the right-hand side is over all integer partitions $\nu$ satisfying $\nu_i\leq\la_i$ for all $i$.
\end{theorem}

\subsection{A combinatorial formula } 
As a consequence in combinatorics, we derive the following by setting $q=1/p$ and switching $\mu$ into $\mu'$ in Theorem \ref{function_gen}, and finally using \eqref{*} and \eqref{**}.
\begin{corollary} For a positive integer $\ell$, let $\lambda=1^{m_1}2^{m_2}\cdots \ell^{m_\ell}$ be a fixed integer partition, and $z$ be a complex number. We have: 
\begin{equation}\label{combinatcl,m}
\sum_{\mu} \frac{q^{\sum_{i\geq1}\mu_i^2}q^{- m_1\mu_1-\cdots -m_\ell(\mu_1+\cdots+\mu_\ell)}}{b_\mu(q)} z^{|\mu|} =  \frac{1}{(zq)_\infty} \sum_{\nu \subseteq \lambda} C_{\lambda,\nu}(1/q) \; z^{|\nu|}.
\end{equation}
\end{corollary} 
\noindent
This identity generalizes \cite[corollary~4]{delaunay3}, and we give here an alternative proof through Lascoux's formula \eqref{lascoux}, without using the previous computations for finite abelian groups. First, it is possible to rewrite the left-hand side of \eqref{combinatcl,m} with the notations of Hall-Littlewood functions. Indeed, replace $\mu$ by its conjugate $\mu'$ on the left-hand side, note that 
\begin{eqnarray*}
\sum_{i=1}^{\ell}m_i(\mu'_1+\dots+\mu'_i)&=&\sum_{i=1}^{\ell}(\la'_i-\la'_{i+1})(\mu'_1+\dots+\mu'_i)\\
&=&\sum_{i=1}^{\ell}\la'_i\mu'_i=(\la'|\mu'),
\end{eqnarray*}
and write $$\sum_{i\geq1}\mu_i'^{2}=2n(\mu)+|\mu|.$$
Then, using the limit case $n\to\infty$ of the  specialization \eqref{principale}, we can rewrite \eqref{combinatcl,m} in the following way:
\begin{equation}\label{combinatcl,mHL}
\sum_{\mu}q^{|\mu|+n(\mu)-(\la'|\mu')}P_\mu(z,zq\dots;q)  =  \frac{1}{(zq)_\infty} \sum_{\nu \subseteq \lambda} C_{\lambda,\nu}(1/q) \; z^{|\nu|}.
\end{equation}
Now multiply both sides of \eqref{combinatcl,mHL} by $q^{n(\la)}P_\la(x;q)$ and sum over all partitions $\la$. As Hall-Littlewood functions form a $\mathbb{Z}[q]$-basis of $\Lambda[q]$, it is enough to show that
\begin{multline}\label{equivalent}
\sum_{\la,\mu}q^{n(\la)+n(\mu)-(\la'|\mu')}P_\la(x;q)P_\mu(zq,zq^2\dots;q) \\
 =  \frac{1}{(zq)_\infty} \sum_{\nu \subseteq \lambda} q^{n(\la)}P_\la(x;q)C_{\lambda,\nu}(1/q) \; z^{|\nu|}.
\end{multline}
Set $y=\{zq,zq^2,\dots\}$ in Warnaar's special case \eqref{warnaara2hl} of Lascoux's identity to see that the left-hand side of \eqref{equivalent} is equal to 
$$\frac{1}{(zq)_\infty}\prod_{i\geq1}\frac{1}{(1-x_i)(1-zx_i)},$$
which, through the generating function \eqref{fgcla,mu}, is easily seen to be equal to the right-hand side of \eqref{equivalent}.
\medskip
~\\

\subsection{Consequences on $u$-averages and proof of Theorem~\ref{coherent}}
Using Theorem~\ref{function_gen} above, replacing $z$ by $p^{-u}$, and recalling the definitions of $u$-averages, we obtain the following $u$-average on finite abelian $p$-groups and on $p$-groups of type~$S$.
\begin{theorem}\label{umoy}
For a positive integer $\ell$, let $\lambda=1^{m_1}2^{m_2}\cdots \ell^{m_\ell}$ be an integer partition. Then, the $u$-average of the function $x^\lambda \colon H \mapsto |H[p]|^{m_1} |H[p^2]|^{m_2} \cdots |H[p^\ell]|^{m_\ell}$ is equal to
$$
M_u(x^\lambda) = \sum_{\mu \subseteq \lambda} C_{\lambda,\mu}(p) p^{-|\mu| u},
$$
and the $u$-average of the function $x^\lambda$ in the sense of groups of type $S$ is equal to
$$
M_u^S(x^\lambda) = \sum_{\mu \subseteq \lambda} C_{\lambda,\mu}(p^2) p^{-|\mu|(2u-1)}.
$$
\end{theorem}
 Theorem \ref{miroir} implies directly that if $\lambda=1^{m_1}2^{m_2}\cdots \ell^{m_\ell}$ is an integer partition, then we have
$$
\mszero(x^\lambda) = \msun(x^\lambda)\,p^{m_1+2m_2+\cdots+ \ell m_\ell},
$$
which is exactly Theorem \ref{coherent}.



\section{Consequence on the heuristics on class groups and on Tate-Shafarevich groups}\label{philo}
 
\subsection{Class groups of quadratic number fields} The Cohen-Lenstra heuristics~(\cite{cohen-lenstra}) 
allow to give precise conjectures for the behavior of the 
class groups of number fields varying in natural families. Here, we only recall briefly the Cohen heuristics in the case of families of quadratic number fields.
For $K_d=\Q(\sqrt{d})$ a quadratic number field with discriminant $d$, we denote by $\cl(d)$ its class group so that $\cl(d)_p$ is its $p$-part. First, take the 
family of imaginary quadratic number fields: these are the fields $K_d$ where $d<0$  runs through the whole set of fundamental negative discriminants. 
Assume that $p\geq 3$ (actually, the $2$-part of the class groups of quadratic fields behaves in a
specific and well-understood way, due to the genus theory, \cite{cohen0}). If $h$ is a "reasonable" function defined over finite abelian $p$-groups, then 
the Cohen-Lenstra heuristics predict that the average of $h$ over the $p$-part of the class group $\cl(d)_p$ is the $0$-average of $h$ ($0$ being the rank of the group of units of $K_d$), that is, we should have
$$
\lim_{X \rightarrow \infty} \frac{\displaystyle  \sum_{|d|\leq X,\, d<0} h(\cl(d)_p)}{\displaystyle \sum_{|d|\leq X,\, d<0} 1} = \mzero(h).
$$ 
For the family of real quadratic number fields, i.e., fields $K_d$ where $d>0$ runs through the set of fundamental positive discriminants, 
the Cohen-Lenstra heuristics predict that the average of $h$ over the $p$-part of the class group $\cl(d)_p$ is the $1$-average of $h$ ($1$ being, in that case,  the rank of $U(K)$), that is, we should have:
$$
\lim_{X \rightarrow \infty} \frac{\displaystyle  \sum_{|d|\leq X,\, d>0} h(\cl(d)_p)}{\displaystyle \sum_{|d|\leq X,\, d>0} 1}  = \mun(h).
$$  
There are strong evidence and numerical data supporting the heuristic principle, however only few results are proved. The first one is for $p=3$ and the 
function $h(H)=|H[3]|$, the number of 3-torsion elements in $H$. Davenport and Heilbronn (\cite{davenport_heilbronn}) proved (before the heuristics were 
formulated) that the average over the 3-part of the class groups of the function $h$ is equal to 2 (resp. 4/3) in the case of imaginary (resp. real) quadratic 
fields. In their paper, Cohen and Lenstra proved (it is not a trivial result) that $\mzero(h)=2$ and $\mun(h)=4/3$. There are other results in the 
literature for other families of number fields, and for the $4$-part of class groups of quadratic number fields (\cite{bhargava}, \cite{kluners_fouvry}).\medskip
~\\
Of course, the case $p=3$ and the function $H \mapsto |H[3]|$ is a particular case of our work. Now, from our computations, we can predict the average 
of all the moments of the number of $p^\ell$-torsion elements in the class groups. 

\begin{conjecture} For any positive integer $\ell$, let $\lambda=1^{m_1}2^{m_2}\cdots \ell^{m_\ell}$ be an integer partition, and assume that $p \geq 3$.~\\
As $d$ is varying over the set of fundamental negative discriminants the average of $|\cl(d)[p]|^{m_1} |\cl(d)[p^2]|^{m_2} \cdots |\cl(d)[p^\ell]|^{m_\ell}$ is equal to
$$
\sum_{\mu \subseteq \la} C_{\la,\mu}(p).
$$
As $d$ is varying over the set of fundamental positive discriminants the average of $|\cl(d)[p]|^{m_1} |\cl(d)[p^2]|^{m_2} \cdots |\cl(d)[p^\ell]|^{m_\ell}$ is equal to
$$
\sum_{\mu \subseteq \la} C_{\la,\mu}(p) p^{-|\mu|}.
$$
\end{conjecture}
In particular, (still for $p\neq2$) and using equation \eqref{exemple_poonen}, one obtains that the average of $|\cl(d)[p]|^n$ of class group of imaginary 
(resp. real) quadratic fields is equal to
$$
M_0(x^{1^n}) = \sum_{k=0}^n \qbi{n}{k}{p} \quad \left(\mbox{resp. } M_1(x^{1^n}) = \frac{1}{p}\sum_{k=0}^n \qbi{n}{k}{p} \right) 
$$
Those are exactly the $\mathcal{N}(n,p)$ (resp. $\mathcal{M}_n(p)$) used in \cite{kluners_fouvry, kluners_fouvry2}.  In particular, Fouvry and Klueners used 
$\mathcal{N}(n,p)$ (resp. $\mathcal{M}_n(p)$) in \cite{kluners_fouvry2} in order to prove that conjectures (coming from a special case of the 
Cohen-Lenstra heuristics) related to the moments of the number of $p$-torsion points in class groups imply conjectures 
(coming from another special case of the Cohen-Lenstra heuristics) about the probability laws of the $p$-ranks of the class groups. 
Our results could be used to generalize their results and also to adapt them in the case of Tate-Shafarevich groups. 
This will be explained in a forthcoming paper. 

\subsection{Tate-Shafarevich groups of elliptic curves}\label{heur_sha}
Let $u\geq 0$ be a fixed integer, and consider the family $\mathcal{F}$ of elliptic curves $E$ defined over $\Q$, with rank $u$, and ordered by the 
conductor $N_E$. The Tate-Shafarevich group $\Sha(E)$ is conjecturally a finite abelian group. If so, then $\Sha(E)$ is a group of type $S$. The 
heuristic principle (\cite{delaunay1, delaunay2}) 
asserts that if $g$ is
 a "reasonable" function defined over isomorphism classes of $p$-group of type $S$, we should have
\begin{equation}\nonumber 
\lim_{X\rightarrow \infty} \frac{\displaystyle \sum_{E \in {\mathcal F},\,N_E\leq X} g(\Sha(E)_p)}{\displaystyle \sum_{E \in {\mathcal F},\,N_E\leq X} 1} = \ms(g).
\end{equation}
\noindent
{\bf Important remark and correction of the first version of the heuristics on Tate-Shafarevich groups.} In this paper, the heuristics on 
Tate-Shafarevich groups, in particular 
the equation above and Conjecture~\ref{heuristics}  take into account the following correction. Due to a bad choice of 
a parameter, 
the heuristic assumption for Tate-Shafarevich groups stated in \cite{delaunay1} is wrong and should be corrected for $u>0$ by the 
following: in the Heuristic Assumption of \cite[page 195]{delaunay1}, the equality {\em $M_u(f)=M_{u/2}^s(f)$ must be replaced by 
$M_u(f)= M_u^s(f)$} (hence, the examples given should also be modified in consequence). Note that this was 
done  in \cite{delaunay2} but  only explicitly for rank 1 elliptic curves. In fact,
this correction was suggested by Kowalski observing that only the corrected version of the heuristics is compatible with 
results of 
Heath-Brown (\cite{heath-brown1,heath-brown2}) on the 2-rank of the Selmer groups of the family of elliptic curves 
$E_d \colon y^2=x^3-d^2x$, $d$ odd and squarefree 
(see \cite[Remark 2.6]{kowalski}). Note that this corrected version is also compatible with other theoretical results and conjectural works, see the 
discussion after Conjecture~\ref{heuristics}.   
\bigskip
~\\
The heuristic principle has many applications and there are  numerical evidence supporting them. With our previous computations on $u$-averages for groups of type $S$ we are led to conjecture the following.
\begin{conjecture}\label{heuristics} For any positive integer $\ell$, let $\lambda=1^{m_1}2^{m_2}\cdots \ell^{m_\ell}$ be an integer partition, and let $u\geq 0$ be an integer. As $E/\Q$, ordered by conductors, is varying over elliptic curves with rank $u$, the average of $|\Sha(E)[p]|^{m_1} |\Sha(E)[p^2]|^{m_2} \cdots |\Sha(E)[p^\ell]|^{m_\ell}$ is equal to
$$
\sum_{\mu \subseteq \lambda} C_{\lambda,\mu}(p^2)p^{-|\mu|(2u-1)}.
$$
\end{conjecture}
\noindent
In particular, if we take the function $g$ such that $g(G)=|G[p^\ell]|$ then we have:
$$
\ms(g)= 1 + \frac{1}{p^{2u-1}} + \cdots + \frac{1}{p^{(2u-1)\ell}},
$$
which for instance,  when $\ell=1$, gives $\mszero(g)=1+p$ and $\msun(g)=1+1/p$.~\\
The rank conjecture in its strong form (asserting that the rank of $E$ is $0$ or $1$ with probability $1/2$ each and that elliptic curves with rank $\geq 2$ 
do not contribute in the averages), and the results of Bhargava and Shankar 
\cite{bhargava_shankar_1, bhargava_shankar_2}, imply that the average size over all elliptic curves $E/\Q$ with rank $0$ (resp. rank $1$) of $\Sha(E)[p]$ is 
$1+p$ (resp. $1+1/p$) for $p=2$ and $p=3$. From their work, Bhargava and Shankar also made a conjecture for all $p$ which, together with the rank conjecture, would imply that the average size $|\Sha(E)[p]|$ is $1+p$ (resp. $1+1/p$) for rank $0$ (resp. rank~$1$) elliptic curves.
~\\
Swinnerton-Dyer \cite{swinnerton-dyer} and Kane \cite{kane} obtained results for $p=2$ and for the family of quadratic twists, $E_d$, of any fixed elliptic curve $E/\Q$ with $E[2] \subset E(\Q)$ and having no rational cyclic subgroup of order $4$. Their results and the rank conjecture imply that the average of 
$|\Sha(E_d)[2]|$ is $3$ (resp. $3/2$) for the rank 0 (resp. rank 1) curves. There are also in the literature some other theoretical results for other families 
of elliptic curves, mainly for $p=2$. \medskip
~\\

Actually, all the previous theoretical results are directly concerned with the behavior of  $p$-Selmer groups of elliptic curves. If $n\in \N$, the 
$n$-Selmer group 
$\sel_n(E)$ and the Tate-Shafarevich group of an elliptic curve are linked by the exact sequence
$$
0 \rightarrow E(\Q)/nE(\Q) \rightarrow \sel_n(E) \rightarrow \Sha(E)[n] \rightarrow 0.
$$
So, if $E(\Q)_{\rm{tors}}=0$ and if the rank of $E$ is $u$, then we have $|\sel_n(E)| = n^u |\Sha(E)[n]|$. Therefore, assuming the rank conjecture, any result on $\sel_n(E)$ gives an information on $\Sha(E)[n]$. Reciprocally, the distribution of $\sel_n(E)$ can be deduced from the distribution of $\Sha(E)[n]$ for rank $0$ and rank $1$ curves.  
\medskip
~\\
Recently, Poonen and Rains \cite{poonen-rains} gave a very deep model for the behavior of the $p$-Selmer groups of all elliptic curves (they do not 
separate elliptic curves by their rank). In particular, they obtained the following predictions for elliptic curves defined over $\Q$. 
\begin{conjecture}[Poonen-Rains] \label{poonen-rains}
For any integer $m\geq 0$, the average of $|\sel_p(E)|^m$ over all $E/\Q$ is equal to $\prod_{j=1}^m (1+p^j)$.
\end{conjecture}
Assuming that the $p$-parts of $\sel_n(E)$ behave independently for $p|n$, except for the constraint that their parities are equal, they obtained a conjecture for the average of $|\sel_n(E)|^m$ for all squarefree 
positive integer $n$.  
\noindent
The Poonen and Rains model suggests that the $p$-Selmer groups behave in the same way for rank $0$ and rank~$1$ elliptic curves. If we assume that elliptic curves of rank $\geq 2$ do not contribute to the average value, then this is also suggested by the heuristic principle since $E(\Q)_{\rm{tors}} =0$ with probability $1$ and $\mszero(g)=\msun(g)p$ (recall that here, $g$ denotes the function $G \mapsto |G[p]|$). It is then natural to ask wether this is still true for all integers $n$, without the squarefree restriction. Theorem \ref{coherent} and the discussions above suggest that this is indeed the case, yielding the following conjecture.
\begin{conjecture}\label{selmer}
Let $\ell\geq 1$ and $m \geq 0$ be integers. The average of $|\sel_{p^\ell}(E)|^m$ over all $E/\Q$ is equal to
$$
\sum_{\mu \subseteq \lambda} C_{\lambda,\mu}(p^2) p^{|\mu|}
$$
where $\lambda$ is the integer partition $\lambda = \ell^m$. 
\end{conjecture}
\noindent
For $\ell=1$, equation \eqref{exemple_poonen} gives $\sum_{\mu \subseteq \lambda} C_{\lambda,\mu}(p^2) p^{|\mu|} = \sum_{k=0}^m \qbi{m}{k}{p^2} p^k$. 
It can be easily shown that the last sum is indeed equal to $\prod_{j=1}^m (1+p^j)$, as claimed in Conjecture \ref{poonen-rains}.\\
~\\
One can clarify the links between the different conjectures involved in the discussion above. If $E$ is an elliptic curve over $\Q$, we denote by $\rk(E)$ the rank of its Mordell-Weil group. We have in mind:
\begin{itemize}
\item The rank conjecture asserting that $\rk(E)$ is 0 or 1 with probability $1/2$ each.
\item The assumption saying that elliptic curves of rank $\geq 2$ contribute nothing to the average value of $p^{m\rk(E)}$. Note that this depends on $m$, therefore we denote by $e_{p,m}$ the lim sup, as $N \rightarrow \infty$, of the sum of $p^{m\rk(E)}$ over curves of rank $\geq 2$ and conductor $\leq N$, divided by the total number of curves of conductor $\leq N$. 
\item The Poonen-Rains conjecture (Conjecture 1.1 in \cite{poonen-rains}). 
\item The heuristic model for the probability laws $\dim_{\F_p} \Sha(E)[p]$ (\cite{delaunay1,delaunay2}).
\item The heuristic model for the moment of $\Sha(E)$, i.e., Conjecture \ref{heuristics} for an integer partition $\lambda$.
\end{itemize}
~\\
As discussed above and also explained in \cite{poonen-rains}, the rank conjecture and Conjecture~1.1 in \cite{poonen-rains} imply  
the heuristic model for $\dim_{\F_p} \Sha(E)[p]$, and so Conjecture~\ref{heuristics} for $|\Sha(E)[p]|^m$ (i.e., $\lambda=1^m$) for rank 0 and rank 1 curves. 
\\
Also, Poonen and Rains proved (\cite[Theorem 5.2]{poonen-rains}) that the rank conjecture is the only distribution on $\rk(E)$ compatible with the Poonen-Rains conjecture 
and the heuristic model for $\dim_{\F_p} \Sha(E)[p]$.
\\
Furthermore, Conjecture~1.1 in \cite{poonen-rains} also implies that $e_{p,m}=0$ for $m=1,2,3$ (in fact, Poonen and Rains prove it for $m=1$ but their argument can be directly adapted for $m=2$ and $m=3$).
\\
On the other hand, one can see that  the equality
$e_{p,m}=0$, the rank conjecture, and Conjecture \ref{heuristics} for the moments of $|\Sha(E)[p^\ell]|^m$ (strengthening with the assumption that the error terms implied 
in the averages are uniformly bounded independently of the rank $u$), imply Conjecture~\ref{selmer}. Indeed (for simplicity we restrict the following computations to elliptic curves $E$ such that $E(\Q)_{\rm{tors}}=\{0\}$), we first see that
\begin{eqnarray*}
\displaystyle \sum_{ \substack{N_E \leq X \\ \rk(E)\geq 2}} |\sel_{p^\ell}(E)|^m &=&\displaystyle \sum_{u\geq 2} p^{mu}\,\frac{\displaystyle \sum_{\substack{N_E \leq X \\ \rk(E)=u}} |\Sha(E)[p^\ell]|^m}{\displaystyle \sum_{\substack{N_E \leq X \\ \rk(E)=u}}1} \displaystyle \sum_{\substack{N_E \leq X \\ \rk(E)=u}}1.
\end{eqnarray*}  
Now, the fraction in the sum of the right-hand side tends to $\ms(x^{\ell^m})$ as $X \rightarrow \infty$. Using the fact that $u \mapsto \ms(x^{\ell^m})$ is decreasing (because the coefficients $C_{\lambda,\mu}$ are positive) and the  above assumption on the error terms, we deduce that this fraction is uniformly bounded (and the bound does not depend on $u$). Hence,
 $$
 \displaystyle \sum_{ \substack{N_E \leq X \\ \rk(E)\geq 2}} |\sel_{p^\ell}(E)|^m = O\left( \displaystyle \sum_{ \substack{N_E \leq X \\ \rk(E)\geq 2}} p^{m\rk(E)}\right).
 $$ 
Therefore, since by assumption $e_{p,m}=0$, we deduce that the average of $|\sel_{p^\ell}(E)|^m$ of curves $E$ of rank $\geq 2$ among all elliptic curves is equal to 0. Now, we can write 
\begin{multline*}
\frac{\displaystyle \sum_{N_E \leq X} |\sel_{p^\ell}(E)|^m}{\displaystyle \sum_{N_E \leq X} 1}   = 
\frac{\displaystyle \sum_{\substack{N_E \leq X \\ \rk(E) = 0}} |\sel_{p^\ell}(E)|^m}{\displaystyle \sum_{\substack{N_E \leq X \\ \rk(E) = 0}}1} \cdot \frac{\displaystyle \sum_{\substack{N_E \leq X \\ \rk(E) = 0}}1}{\displaystyle \sum_{N_E \leq X} 1}  \\
 + \frac{\displaystyle \sum_{\substack{N_E \leq X \\ \rk(E) = 1}} |\sel_{p^\ell}(E)|^m}{\displaystyle \sum_{\substack{N_E \leq X \\ \rk(E) = 1}}1} \cdot \frac{\displaystyle \sum_{\substack{N_E \leq X \\ \rk(E) = 1}}1}{\displaystyle \sum_{N_E \leq X} 1} 
 + \frac{\displaystyle \sum_{\substack{N_E \leq X \\ \rk(E) \geq 2}} |\sel_{p^\ell}(E)|^m}{\displaystyle \sum_{N_E \leq X} 1}.
\end{multline*}
The last term on the right-hand side tends to 0 as $X \rightarrow \infty$. Now, replace $|\sel_{p^\ell}|$ by $|\Sha(E)[p^\ell]|$  in the first term and $|\sel_{p^\ell}|$ by $p^\ell |\Sha(E)[p^\ell]|$ in the second one and use Conjecture \ref{heuristics}, Theorem \ref{coherent} and the rank conjecture to conclude.
\bigskip
~\\
{\bf Remark.} In the above discussions, one can replace the number field $\Q$ by an arbitrary number field $K$.


\section{$u$-averages and Hall-Littlewood polynomials}\label{combinatorial_consequence}

\subsection{$u$-averages and combinatorial consequences}

In the paper \cite{delaunay3}, the first author computed the $u$-probabilities laws for the $p^j$-ranks of finite abelian $p$-groups (recall that the $p^j$-rank of $H$ is the dimension over $\F_p$ of the vector space $p^{j-1}H/p^{j}H$ and is denoted $r_{p^j}(H)$), which we rewrite below.
\begin{proposition}[Corollary~11 of \cite{delaunay3}]\label{pj-ranks} Let $\ell\geq 1$ be an integer and $\mu$ be an integer partition with $\ell(\mu)\leq \ell$ (here $\mu$ is given by its parts $\mu_1 \geq \mu_2\geq \cdots \geq \mu_\ell \geq 0$). Then the $u$-probability that a finite abelian $p$-group has its $p^j$-rank equal to $\mu_j$ for all $1\leq j \leq \ell$ is equal to
$$
\frac{\prod_{j\geq \mu_\ell+1} (1-1/p^{u+j})}{p^{\mu_1^2+\cdots + \mu_\ell^2+u(\mu_1+\cdots + \mu_\ell)} \prod_{j=1}^{\ell} (1/p;1/p)_{\mu_j-\mu_{j+1}}},
$$
and the $u$-probability that a group of type $S$ has its $p^j$-rank equal to $2\mu_j$ for all $1\leq j \leq \ell$ is equal to
$$
\frac{\prod_{j\geq \mu_\ell+1} (1-1/p^{2u+2j-1})}{p^{2(\mu_1^2+\cdots + \mu_\ell^2)+(2u-1)(\mu_1+\cdots + \mu_\ell)} \prod_{j=1}^{\ell} (1/p^2;1/p^2)_{\mu_j-\mu_{j+1}}}.
$$
\end{proposition}
Hence, for a fixed integer partition $\lambda=1^{m_1}2^{m_2}\cdots \ell^{m_\ell}$, we can also compute the $u$-average of $x^\lambda$ over finite abelian $p$-groups (with the notations of Theorem~\ref{umoy}) by its moments. Indeed, we have
$$
x^\lambda(H) = p^{r_p(H)m_1 + (r_p(H)+r_{p^2}(H))m_2 + \cdots + (r_p(H)+\cdots + r_{p^\ell}(H))m_\ell},
$$
therefore 
$$
M_u(x^\lambda) = \sum_{0\leq \mu_\ell \leq \cdots \leq \mu_1} \uprob(r_p(H)=\mu_1, \dots, r_{p^\ell}(H)=\mu_\ell) p^{\mu_1m_1 +  \cdots + (\mu_1+\cdots + \mu_\ell)m_\ell}.
$$
Replacing $\uprob(r_p(H)=\mu_1, \dots, r_{p^\ell}(H)=\mu_\ell)$ by the value given in Proposition~\ref{pj-ranks} and $M_u(x^\lambda)$ by its value  from Theorem \ref{umoy}, we obtain
$$
\sum_{\ell(\mu)\leq \ell} \frac{\prod_{j\geq \mu_\ell+1} (1-1/p^{u+j})}{p^{\mu_1^2+\cdots + \mu_\ell^2+u(\mu_1+\cdots + \mu_\ell)} \prod_{j=1}^{\ell} (1/p;1/p)_{\mu_j-\mu_{j+1}}} p^{(\lambda'|\mu)} = \sum_{\nu \subseteq \lambda} C_{\lambda,\nu}(p)p^{-|\nu|u}.
$$
~\\
By setting $q =1/p$, replacing $\mu$ by its conjugate $\mu'$ on the left-hand side, and setting $z=q^{-u}$, we see that this formula is equivalent to the following combinatorial identity.
\begin{theorem}
For any nonnegative integer $\ell$, let $\lambda=1^{m_1}2^{m_2}\cdots \ell^{m_\ell}$ be a fixed integer partition such that $\la_1\leq\ell$ (equivalently $\ell(\la')\leq\ell$), and $z$ be a complex number. Then we have
\begin{equation}\label{umoyabeliens}
\sum_{\mu_1\leq \ell} \frac{ z^{|\mu|}q^{2n(\mu)+|\mu|-(\la'|\mu')}}{b_\mu(q)}(zq^{\mu'_\ell+1})_\infty  = \sum_{\nu \subseteq \lambda} C_{\lambda,\nu}(1/q) z^{|\nu|}.
\end{equation}
\end{theorem}
\medskip
~\\
{\bf Remark.} One can also replace finite abelian $p$-groups by groups of type S in the previous discussion. We obtain as a consequence the same combinatorial identity as \eqref{umoyabeliens}, in which $q$ is replaced by $q^2$ and then $z$ is replaced by $z^2/q$.\\
~\\
Note that thanks to Theorem~\ref{miroir}, the sum $\sum_{\nu \subseteq \lambda} C_{\lambda,\nu}(1/q) q^{-|\nu|u}$ must be equal to $q^{-|\lambda|u} \sum_{\nu \subseteq \lambda} C_{\lambda,\nu}(1/q) q^{|\nu|u}$. By using \eqref{umoyabeliens} with $z=q^{-u}$ and $z=q^u$, this yields an identity between the left-hand sides which seems quite mysterious. \\
~\\
In the case where $m_1=m_2=\cdots=m_{\ell-1}=0$ and $m_\ell=1$ (therefore $\la=(\ell)$, a row partition), identity \eqref{umoyabeliens} corresponds to \cite[Corollary~12]{delaunay3}, which can be written as follows
\begin{equation}\label{delaunayumoy}
\sum_{\mu_1\leq\ell}\frac{z^{|\mu|}q^{2n(\mu)}}{b_\mu(q)}(zq^{\mu'_\ell+1})_{\infty}=\frac{1-z^{\ell+1}}{1-z}.
\end{equation}

\medskip

In view of the structure of \eqref{umoyabeliens}, it would be a challenging problem to prove it through the theory of Hall-Littlewood functions, by using Warnaar's identity \eqref{warnaara2hl} or Lascoux's generalization \eqref{lascoux}, as we did for proving the symmetry property \eqref{conj'}. However, problems occur as we need to bound the first part of the partition $\mu$ on the left-hand side in \eqref{umoyabeliens}, therefore a finite summation of Hall-Littlewood polynomials seems to be required (i.e., a summation over partitions where both the length and the first part are bounded). In his paper \cite{Warnaar}, Warnaar raises the question of finding a finite form of \eqref{warnaara2hl}, which seems out of reach, and therefore finding a finite form of Lascoux's identity seems hopeless. Note that in \eqref{warnaara2hl}, specializing the variables $y_i$ to $0$, one recovers the special case $a=0$ of \eqref{qbinhl}, for which Warnaar proves a finite version in \cite[Theorem~6.1]{Warnaar}. Unfortunately, the latter neither contains \eqref{umoyabeliens} nor  even \eqref{delaunayumoy} as special cases.
Inspired by Macdonald's partial fraction method \cite{Macdonald} used in \cite{stembridge, JZ, IJZ, Warnaar}, we can prove the following finite form of \eqref{qbinhl}, which will be shown to contain \eqref{delaunayumoy} as a special case.
\begin{theorem}
Let $n$ be a positive integer, and $x=\{x_1,\dots,x_n\}$ be a set of $n$ variables. Then for any complex number $a$ and any positive integer $k$, we have
\begin{multline}\label{qbinhlfini}
\sum_{\la\subseteq(k^n)}q^{n(\la)}(a;q^{-1})_{\ell(\la)}(a;q^{-1})_{n-m_k(\lambda)}P_\la(x;q)\\
=\sum_{I\subseteq[n]}q^{k\bi{|I|}{2}}(a;q^{-1})_{|I|}(a;q^{-1})_{n-|I|}\prod_{i\in I}x_i^{k}\\
\times\prod_{i\in I}\frac{1-ax_i^{-1}q^{1-n}}{1-x_i^{-1}q^{1-|I|}}\prod_{j\notin I}\frac{1-ax_j}{1-x_jq^{|I|}}\prod_{i\in I,\,j\notin I}\frac{x_i-qx_j}{x_i-x_j}\,,
\end{multline}
where the sum on the left-hand side is over partitions $\lambda$ satisfying $\lambda_1\leq k$ and $\ell(\lambda)\leq n$, $[n]:=\{1,\dots,n\}$, and $|I|$ is the cardinality of the set $I$.
\end{theorem}
We postpone the proof of this result to the next subsection, but we give the following usefull consequence.
\begin{corollary}
For positive integers $n$, $k$, and for any complex numbers $a$, $z$, we have
\begin{multline}\label{csqqbinhlfini}
\sum_{\la\subseteq(k^n)}z^{|\la|}q^{2n(\la)}\frac{(a;q^{-1})_{\ell(\la)}(a;q^{-1})_{n-m_k(\la)}(q)_n}{(q)_{n-\ell(\lambda)}b_\la(q)}=\sum_{r=0}^n(-1)^rz^{(k+1)r}q^{(2k+3)\bi{r}{2}}\\
\times(1-zq^{2r-1})\frac{(q^{n-r+1})_r(azq^r)_{n-r}(a;q^{-1})_r(aq^{1-n}/z;q^{-1})_{r}(a;q^{-1})_{n-r}}{(q)_r(zq^{r-1})_{n+1}}\,\cdot
\end{multline}
\end{corollary}
\begin{proof}
We consider the  specialization $x_i=zq^{i-1}$ in \eqref{qbinhlfini}. In that case 
note that for any $I\subseteq[n]$,
$$\prod_{i\in I,j\notin I}\frac{1-qx_i^{-1}x_j}{1-x_i^{-1}x_j}\neq 
0\;\Leftrightarrow\;\exists 
r\in\{0,\cdots,n\}\,|\,I=\{1,\dots,r\}.
$$
Thus, by using \eqref{principale}, we get the desired result after a few manipulations on the right-hand side.
\end{proof}
Note that when $n\to\infty$, identity (\ref{csqqbinhlfini}) reduces to Stembridge's Theorem~3.4 (b) with $b=0$ 
in \cite{stembridge}. Moreover, when $a=0$, (\ref{csqqbinhlfini}) corresponds to Warnaar's identity (6.10) from \cite{Warnaar}, which is itself equivalent to Stembridge's Theorem~1.3~(b) from \cite{stembridge}. Now we can easily see that \eqref{csqqbinhlfini} contains \eqref{delaunayumoy} (but unfortunately not \eqref{umoyabeliens}) as a consequence. Indeed, on the left-hand side of \eqref{csqqbinhlfini}, replace $\la$ by $\mu$, $k$ by $\ell$, and take $a=zq^n$ to get
\begin{multline*}
\sum_{\mu_1\leq\ell}z^{|\mu|}q^{2n(\mu)}\frac{(zq^n;q^{-1})_{l(\mu)}(zq^n;q^{-1})_{n-\mu'_\ell}(q)_n}{(q)_{n-l(\mu)}b_\mu(q)}=\sum_{r=0}^n(-1)^rz^{(\ell+1)r}q^{(2\ell+3)\bi{r}{2}}\\
\times(1-zq^{2r-1})\frac{(q^{n-r+1})_r(z^2q^{r+n})_{n-r}(zq^n;q^{-1})_r(q;q^{-1})_{r}(zq^n;q^{-1})_{n-r}}{(q)_r(zq^{r-1})_{n+1}}\,\cdot
\end{multline*}
Note that on the right-hand side, $(q;q^{-1})_{r}=0$ unless $r=0$ or $1$, therefore letting finally $n\to\infty$ yields \eqref{delaunayumoy}.

\subsection{Proof of the finite form of the $q$-binomial theorem for Hall-Littlewood polynomials}

Consider the generating function $$
S(u)=\sum_{\la_0,\la}q^{n(\la)}c_{\la_0}(\la,a)P_\la(x;q)\,
u^{\la_0}, $$ where the sum is over all partitions
$\la=(\la_1,\ldots, \la_n)$ and the integers
 $\la_0\geq \la_1$, and $$
c_k(\la,a)=(a;q^{-1})_{l(\la)}(a;q^{-1})_{n-m_k(\la)}.
$$
Assume $\la=(\mu_1^{r_1}\, \mu_2^{r_2}\, \ldots \mu_k^{r_k})$,
 where
$\mu_1>\mu_2>\cdots >\mu_k\geq 0$ and $(r_1, \ldots, r_k)$ is a
 composition of $n$. Let $S_n^\la$ be the set of permutations of $S_n$ which fix $\la$.
 Each $w\in S_n/S_n^\la$ corresponds to  a surjective mapping $f:
x\longrightarrow \{1,2,\ldots, k\}$ such that $|f^{-1}(i)|=r_i$.
For any subset $y$ of $x=\{x_1,\dots,x_n\}$, let $p(y)$ denote  the product of the
elements of $y$ (in particular, $p(\emptyset)=1$). We can rewrite
Hall-Littlewood functions as follows 
$$ 
P_\la(x;q)=\sum_{f}
p(f^{-1}(1))^{\mu_1}\cdots p(f^{-1}(k))^{\mu_k}
\prod_{f(x_i)<f(x_j)}{x_i-qx_j\over x_i-x_j}, 
$$ 
summed over all
surjective mappings $f: x\longrightarrow \{1,2,\ldots, k\}$ such
that $|f^{-1}(i)|=r_i$. Furthermore, each such $f$ determines a
\emph{filtration}  of $x$
\begin{equation}\label{filtration}
 {\mathcal F}: \quad
\emptyset=F_0\subsetneq F_1\subsetneq \cdots \subsetneq F_k=x,
\end{equation}
according to the rule $x_i\in F_l\Longleftrightarrow f(x_i)\leq l$
for $1\leq l\leq k$. Conversely, such a filtration ${\mathcal
F}=(F_0,\, F_1, \ldots, F_k)$ determines a surjection $f:
x\longrightarrow \{1,2,\ldots, k\}$ uniquely. Thus we can write
\begin{equation}\label{filtre}
P_\la(x;q)=\sum_{\mathcal F}\pi_{\mathcal F}\prod_{1\leq i\leq
k}p(F_i\setminus F_{i-1})^{\mu_i},
\end{equation}
summed  over all the filtrations $\mathcal F$  such that
 $|F_i|=r_1+r_2+\cdots +r_i$ for $1\leq i\leq k$, and
$$ \pi_{\mathcal F}:=\prod_{f(x_i)<f(x_j)}{x_i-qx_j\over x_i-x_j}, $$
where $f$ is the function defined by $\mathcal F$. Now let $\nu_i=\mu_i-\mu_{i+1}$ if $1\leq i\leq k-1$ and
$\nu_k=\mu_k$, thus $\nu_i>0$ if $i<k$ and $\nu_k\geq 0$. 
Furthermore, let  $\mu_0=\la_0$ and
$\nu_0=\mu_0-\mu_1$ in the definition of $S(u)$, so that
$\nu_0\geq 0$ and $\mu_0=\nu_0+\nu_1+\cdots +\nu_k$. Since the  lengthes of the columns of $\la$ are
 $|F_j|=r_1+\cdots +r_j$
with multiplicities $\nu_j$ for $1\leq j\leq k$, we have
$$n(\la)=\sum_{i=1}^{k}\nu_i\bi{|F_i|}{2},$$
$$\ell(\la)=n\chi(\nu_k\neq 0)+|F_{k-1}|\chi(\nu_k=0),$$
$$m_k(\la)=\chi(\nu_0\neq 0)+|F_1|\chi(\nu_0=0),$$
 and 
\begin{eqnarray*}
&&c_{\la_0}(\la,a)=(a;q^{-1})_n\left(\chi(\nu_k\neq 
0)+\frac{(a;q^{-1})_{|F_{k-1}|}}{(a;q^{-1})_n}\chi(\nu_k=0)\right)\\
&&\hskip 3 cm\times\left(\chi(\nu_0\neq 
0)+\frac{(a;q^{-1})_{n-|F_1|}}{(a;q^{-1})_n}\chi(\nu_0=0)\right).
\end{eqnarray*}
For any filtration $\mathcal F$ of $x$ define
\begin{eqnarray*}
&&\hskip -1cm{\mathcal A}_{\mathcal F}(x,a,u):=(a;q^{-1})_n\\
&&\hskip -1cm\times\prod_{j=0}^{k}\left({q^{\bi{|F_{j}|}{2}}p(F_j)u \over 
1-q^{\bi{|F_{j}|}{2}}p(F_j)u}+\frac{(a;q^{-1})_{|F_{k-1}|}}{(a;q^{-1})_n}\chi(F_j=x)+\frac{(a;q^{-1})_{n-|F_1|}}{(a;q^{-1})_n}\chi(F_j=\emptyset)\right).
\end{eqnarray*}
Letting $F(x)$ be the set of filtrations of $x$, we get  $S(u)=\sum_{{\mathcal F}\in F(x)}\pi_{\mathcal F}{\mathcal A}_{\mathcal F}(x,a,u).$
 Hence  $S(u)$ is a rational function of $u$ with simple
 poles at  $1/p(y)q^{\bi{|y|}{2}}$, where $y$ is a subset of $ x$. We are
now proceeding to compute  the corresponding  residue $c(y)$
 at each pole $u=1/p(y)q^{\bi{|y|}{2}}$. Let us  start with $c(\emptyset)$. Writing $\la_0=\la_1+k$ with
$k\geq 0$, we see that
\begin{eqnarray*}
S(u)&=&\sum_\la q^{n(\la)}(a;q^{-1})_{l(\la)}P_\la(x;q)u^{\la_1}\sum_{k\geq 
0}u^k\frac{(a;q^{-1})_{n-m_{\la_1+k}}}{(a;q^{-1})_{n}}\\
&=&\sum_\la q^{n(\la)}(a;q^{-1})_{l(\la)}P_\la(x;q)u^{\la_1}\left({u\over 
1-u}+\frac{(a;q^{-1})_{n-m_{\la_1}}}{(a;q^{-1})_{n}}\right).
\end{eqnarray*}
It follows from \eqref{qbinhl} that 
$$
c(\emptyset)=\left[S(u)(1-u)\right]_{u=1}=\prod_{i\geq1}\frac{1-ax_i}{1-x_i}=:\phi(x;a),
$$
which gives the identity
\begin{equation}\label{aempty} 
\sum_{{\mathcal F}\in F(x)}\pi_{\mathcal F}{\mathcal A}_{\mathcal 
F}(x,a,u)(1-u)|_{u=1}=\phi(x;a). \end{equation}
For the computations of other  residues, we need some more
notations. For any $y\subseteq x$ and any $\alpha\in\mathbb{C}$, let $q^\alpha y=\{q^\alpha x_i, x_i\in y\}$, $y'=x\setminus y$, and
$-y=\{x_i^{-1}:x_i\in y\}$. Let $y\subseteq x$, then
\begin{equation}\label{aresidu}
c(y)=\left[\sum_{\mathcal F}\pi_{\mathcal F}{\mathcal A}_{\mathcal
F}(x,a,u)(1-p(y)q^{\bi{|y|}{2}}u)\right]_{u=p(-y)q^{-\bi{|y|}{2}}}.
\end{equation}
If $y\notin\mathcal F$, the corresponding summand is  equal to  $0$.
Thus we need only to consider the following filtrations ${\mathcal F}$
 $$ \emptyset=F_0\subsetneq \cdots \subsetneq F_p=y\subsetneq
\cdots \subsetneq  F_k=x\qquad 1\leq p\leq k. 
$$ 
We may then split
$\mathcal F$ into two filtrations ${\mathcal F}_1$ and ${\mathcal F}_2$
\begin{eqnarray*}
{\mathcal F}_1&:& \emptyset \subsetneq  -q^{-|y|+1}(y\setminus
F_{p-1})\subsetneq \cdots \subsetneq -q^{-|y|+1}(y\setminus F_1)\subsetneq
-q^{-|y|+1}y,\\ {\mathcal F}_2&:& \emptyset \subsetneq q^{|y|}(F_{p+1}\setminus
y)\subsetneq   \cdots \subsetneq q^{|y|}(F_{k-1}\setminus y)\subsetneq  
q^{|y|}y'.
\end{eqnarray*}
As 
$$\displaystyle(|y|-1)(|y|-|F_j|)+\bi{|F_{j}|}{2}-\bi{|y|}{2}=\bi{|y|-|F_j|}{2}\;\mbox{for}\;|y|\geq 
|F_j|,$$
$$\displaystyle 
|y|(|F_j|-|y|)+\bi{|F_{j}|}{2}-\bi{|y|}{2}=\bi{|F_j|-|y|}{2}\;\mbox{for}\;|y|\leq 
|F_j|,$$
and $$\pi_{\mathcal F}(x)=\pi_{{\mathcal F}_1}(-q^{-|y|+1}y)\pi_{{\mathcal 
F}_2}(q^{|y|}y')\prod_{x_i\in y, x_j\in 
y'}\frac{1-qx_i^{-1}x_j}{1-x_i^{-1}x_j},$$ 
we can write, setting $v=p(y)q^{\bi{|y|}{2}}u$,  and using 
(\ref{aresidu}), 
\begin{eqnarray*}
c(y)&=&\left[\sum_{{\mathcal F}\in F(x)}\pi_{\mathcal F}{\mathcal A}_{\mathcal 
F}(x,a,u)\left(1-p(y)q^{\bi{|y|}{2}}u\right)\right]_{u=p(-y)q^{-\bi{|y|}{2}}}\\
&=&\frac{(a;q^{-1})_{n}}{(aq^{|y|-n};q^{-1})_{|y|}(aq^{-|y|};q^{-1})_{n-|y|}}\prod_{x_i\in 
y, x_j\in y'}\frac{1-qx_i^{-1}x_j}{1-x_i^{-1}x_j}\\
&&\hskip 2cm\times \left[\sum_{{\mathcal F}_1}\pi_{{\mathcal F}_1}{\mathcal A}_{{\mathcal 
F}_1}\left(-q^{-|y|+1}y,aq^{|y|-n},v\right)(1-v)\right]_{v=1}\\
&&\hskip 2cm\times \left[\sum_{{\mathcal F}_2}\pi_{{\mathcal F}_2}{\mathcal A}_{{\mathcal 
F}_2}\left(q^{|y|}y',aq^{-|y|},v\right)(1-v)\right]_{v=1}.
\end{eqnarray*}
Then, using (\ref{aempty}), we get
\begin{eqnarray*}
c(y)&=&\phi\left(-q^{1-|y|}y;aq^{|y|-n}\right)\;\phi\left(q^{|y|}y';aq^{-|y|}\right)\\
&&\hskip 
1cm\times\frac{(a;q^{-1})_{|y|}(a;q^{-1})_{n-|y|}}{(a;q^{-1})_{n}}\prod_{x_i\in 
y, x_j\in y'}\frac{1-qx_i^{-1}x_j}{1-x_i^{-1}x_j}.
\end{eqnarray*}
Encoding each subset $y$ of $x$ by the corresponding subset $I\subseteq[n]$ and extracting the coefficient $u^k$ 
in
$$ S(u)=\sum_{y\subseteq x} {c(y)\over 1-q^{\bi{|y|}{2}}p(y)u},$$
we get the result.

\section{Acknowledgements}
We thank B.~Poonen, M.~Schlosser and S.~O.~Warnaar  for helpful discussions and suggestions during the preparation of this manuscript.


%
%

\bibliographystyle{alpha}
\bibliography{pl_ranks_full}

\end{document}